\documentclass[11pt]{article}

\usepackage[T1]{fontenc}
\usepackage[utf8]{inputenc}
\usepackage{lmodern}
\usepackage[a4paper,margin=1.12in]{geometry}
\usepackage{amsmath,amsthm,amssymb,mathtools,mathrsfs}
\usepackage{microtype}
\usepackage{xcolor}
\usepackage{enumerate}

\usepackage[colorlinks=true,linkcolor=blue!55!black,citecolor=blue!55!black,
  urlcolor=blue!55!black]{hyperref}

\numberwithin{equation}{section}

\newtheorem{theorem}{Theorem}[section]
\newtheorem{lemma}[theorem]{Lemma}
\newtheorem{proposition}[theorem]{Proposition}
\newtheorem{corollary}[theorem]{Corollary}
\theoremstyle{definition}
\newtheorem{definition}[theorem]{Definition}
\theoremstyle{remark}
\newtheorem{remark}[theorem]{Remark}
\newtheorem*{remark*}{Remark}

\newcommand{\R}{\mathbb R}
\newcommand{\cN}{\mathcal N}
\newcommand{\cG}{\mathcal G}
\newcommand{\cQ}{\mathcal Q}
\newcommand{\rad}{\mathrm{rad}}
\newcommand{\loc}{\mathrm{loc}}
\newcommand{\dd}{\,dx}
\newcommand{\norm}[1]{\lVert #1\rVert}
\newcommand{\abs}[1]{\lvert #1\rvert}

\title{
    Ground states of a logarithmic Schr\"odinger equation with a
power perturbation in $\R^N$: existence, uniqueness and nondegeneracy}
\author{
		Jinqi Yan \footnote{Email: yjq20011027@163.com} \quad 
		Bo You \footnote{Email: youb2013@xjtu.edu.cn}
		\\
		{\small School of Mathematics and Statistics, Xi'an Jiaotong University} \\
		\small Xi'an, 710049, P. R. China}

\date{}

\begin{document}

\maketitle

\begin{abstract}
The objective of this paper is to study the existence and uniqueness of positive ground state solutions of the following Schr\"odinger problem
\[
 -\Delta u=u\log u^2+\mu\abs{u}^{p-1}u\qquad\text{in }\R^N,
\]
where $N\ge2$, $1<p<2^*-1$ and $\mu>0.$ We prove that there exists a positive ground state solution for every $\mu>0$ and every such solution is smooth, radially symmetric with respect to some point, strictly decreasing in the radial variable. Moreover, as $\mu\downarrow0,$ the positive ground state solutions centered at the origin will converge, uniformly over the ground-state set, to
$U_0(x)=\exp(\frac{(N-\abs{x}^2)}{2})$ in the logarithmic Orlicz--Sobolev space and in $C^{2,\alpha}(\R^N)$ for every $\alpha\in(0,1),$ respectively. Furthermore, for sufficiently small $\mu,$ we also prove that the positive ground state solutions are unique up to translations  and the kernel space of the real scalar linearized operator can be spanned by the translation derivatives. The proof combines minimization on the Nehari manifold with uniform compactness and an exterior coefficient estimate that prevents normalized difference and kernel sequences from vanishing locally.
\end{abstract}



\section{Introduction}
\label{sec:introduction}

In this paper, we mainly consider the Schr\"odinger equation with combined polynomial-logarithmic nonlinearity
\begin{equation}\label{eq:main}
-\Delta u=u\log u^2+\mu\abs{u}^{p-1}u\qquad\text{in }\R^N,
\end{equation}
where $N\ge2,$ $1<p<2^*-1$ and $\mu>0.$ Here $2^*=\frac{2N}{N-2},$ if $N\ge3;$ $2^*=\infty,
$ if $N=2.$ We would like to construct ground state solutions to problem \eqref{eq:main}
 for any $p\in(1,2^*-1)$ and to determine whether the uniqueness and nondegeneracy of the ground state solution to problem \eqref{eq:main} with $\mu=0$ will persist
under a small power perturbation.

As we know, the ground state solutions of nonlinear scalar field equations 
\begin{align}\label{2}
-\Delta u=f(u)\qquad\text{in }\R^N,
\end{align}
have been studied extensively and the existence of ground state solutions is closed related with the behavior of $f$ near zero and at infinity, and a positivity condition on its primitive function $\int_0^uf(s)\,ds.$ For example, in the case of dimension $N\geq 3,$ Berestycki and Lions \cite{BerestyckiLions1983} have proved the existence of a positive radial ground state solution for problem \eqref{2} with continuous odd nonlinearities, which satisfies
\[
-\infty < \liminf_{t\to 0^+}\frac{f(t)}{t}
\le \limsup_{t\to 0^+}\frac{f(t)}{t}
= -m < 0,
\
-\infty \le \limsup_{t\to +\infty}\frac{f(t)}{t^{2^*-1}}\le 0, \
\]
and
\[
\text{there exists a constant } \zeta>0, \text{ such that}
\int_0^\zeta f(s)\,ds>0
.
\]
Moreover, they also proved the uniqueness of such ground state solution under some additional conditions on the nonlinearity. The uniqueness of the ground state solution has been proved in \cite{Coffman1972} for the power model 
\begin{align}\label{3}
-\Delta u+u=u^p\qquad\text{in }\R^3
\end{align}
in the case of $p=3.$ Kwong \cite{Kwong1989} has also established the uniqueness of the positive radial decaying solution for problem \eqref{3} in the case of $p\in (1,5).$ For more general nonlinearities $f$, McLeod and Serrin
\cite{McLeodSerrin1987} and McLeod \cite{McLeod1993} have developed the uniqueness criteria based on comparison arguments for the radial initial-value problem.

The logarithmic term introduces a singular behavior at small amplitudes. Indeed, if $f_\mu(s)=s\log s^2+\mu s^p$ for $s>0$, then we have $\frac{f_\mu(s)}{s}\to-\infty$ as $s\downarrow0$. Thus, the finite negative linear limit assumed in the classical scalar field existence theorem is absent. The logarithmic Schr\"odinger model was introduced by
Bia\l ynicki-Birula and Mycielski \cite{BialynickiBirulaMycielski1976}, and Cazenave \cite{Cazenave1983} developed an appropriate energy-space framework and studied the stability of the Gaussian solutions. As for problem \eqref{eq:main},
the corresponding energy functional is given by
\begin{equation}\label{eq:action-intro}
 J_\mu(u)=\frac12\int_{\R^N}\bigl(\abs{\nabla u}^2+u^2\bigr)\dd
 -\frac12\int_{\R^N}u^2\log u^2\dd
 -\frac{\mu}{p+1}\int_{\R^N}\abs{u}^{p+1}\dd.
\end{equation}
Although the power term can be controlled by the subcritical Sobolev embedding, the negative part of $u^2\log u^2$ need not be integrable for a general $H^1(\R^N)$ function. Thus, we will work in the logarithmic Orlicz-Sobolev space
\begin{align*}
 W=\left\{u\in H^1(\R^N):\int_{\R^N}\abs{u}^2\abs{\log u^2}\dd<+\infty\right\}
\end{align*}
with the convention $0\log 0=0$ and the norm specified in
Definition~\ref{def:energy-space}, such that the functional $J_\mu$ is of class $C^1$ on $W.$

When $\mu=0,$ problem \eqref{eq:main} is reduced to the logarithmic Schr\"odinger model
\begin{equation*}
-\Delta U=U\log U^2\qquad\text{in }\R^N,
\end{equation*}
which admits the explicit Gaussian solution
\begin{equation*}
 U_0(x)=\exp\!\left(\frac{N-|x|^2}{2}\right).
\end{equation*}
d'Avenia, Montefusco, and Squassina
\cite{dAveniaMontefuscoSquassina2014} classified positive decaying solutions of the logarithmic equation in dimensions $N\ge3$ and established the nondegeneracy of the Gaussian. Troy \cite{Troy2016} proved uniqueness of the positive radial ground state in dimensions $1\le N\le9$ by an ODE argument. The two-dimensional classification considered here also follows from the three-dimensional result in \cite[Theorem~1.2]{dAveniaMontefuscoSquassina2014} by multiplying a solution by a one-dimensional Gaussian. Some simple calculations show that
\begin{equation}\label{eq:gaussian-energy}
 \norm{U_0}_2^2=e^N\pi^{\frac N 2},
 \qquad
 c_0=J_0(U_0)=\frac12e^N\pi^{\frac N 2}.
\end{equation}

Several related developments describe the role of logarithmic ground states in perturbation problems. Wang and Zhang \cite{WangZhang2019} proved that suitable rescalings of ground states for power-law scalar field equations converge to a logarithmic ground state as the power approaches the linear exponent. An and Fang \cite{AnFang2025}
used power-law approximations to consider the ground states of problem $-\Delta u=u(\log\abs{u})^\alpha$ in the case that the admissible logarithmic exponents $\alpha\ge1$. They proved the radial uniqueness and nondegeneracy as 
the admissible exponent $\alpha$ tends to $1$, and established the uniqueness up to translations by using the symmetry of positive decaying solutions. Zhang and Zhang \cite{ZhangZhang2022} established the existence, uniqueness and nondegeneracy of positive radial solutions for $-\Delta u+V(\abs{x})u=u\log u^2$ with prescribed radial potential under some suitable assumptions on $V.$ In the present
problem, the logarithmic term is fixed and the parameter $\mu$ controls an additional power nonlinearity. If we interpret this term as a potential $V=-\mu u^{p-1},$ which will
depends on the unknown solution, such that we cannot use the results in \cite{ZhangZhang2022} to consider problem \eqref{eq:main}.

Mixed logarithmic and power nonlinearities also arise in problems with prescribed mass. Shuai and Yang
\cite{ShuaiYang2025} studied the normalized solutions of problem
\[
 -\Delta u+\lambda u=u\log u^2+\mu\abs{u}^{p-1}u,
 \qquad \norm{u}_2=c,
\]
within a broader parameter family. They obtained the existence of normalized ground states and excited states in several exponent regimes, and investigated
the limit as $\mu\downarrow0.$ Here, we prescribe the coefficient $\mu$ in \eqref{eq:main} and minimize the functional among nontrivial solutions.
 These formulations are related by amplitude scaling: setting $v=e^{-\frac{\lambda}{2}}u$ removes the linear term but changes the power coefficient $\mu$ to $\mu e^{\frac{(p-1)\lambda}{2}}$ and also changes the mass.
Our results address uniqueness up to translations and the full linearized kernel for positive ground states at a prescribed small value of $\mu.$ These conclusions require control of the entire set of ground states for the fixed equation.

We establish existence and symmetry for every \(\mu>0\), uniform
 convergence of the entire centered ground-state set as \(\mu\downarrow0\),
  and uniqueness up to translations and full real scalar nondegeneracy for
   sufficiently small \(\mu\). Let \(c_\mu^\cN\) denote the Nehari level and
    \(\mathcal{G}_\mu^{\mathrm{rad}}\) the set of positive ground states radial
     about the origin; the precise definitions are given in Section~\ref{sec:variational}.

The logarithmic energy need not be finite on \(H^1(\mathbb{R}^N)\), 
and the derivative of \(s\log s^2\) is unbounded near zero. We minimize
 \(J_\mu\) on the Nehari constraint in \(W\). Schwarz rearrangement
  reduces the minimization to radial functions, while radial compactness 
  and rescaling recover the constraint at the weak limit. After verifying
   that the constraint is natural, we establish strict positivity by
    a maximum principle adapted to the logarithmic singularity, then
     apply elliptic bootstrapping. The equality case in the P\'olya--Szeg\H{o} 
     inequality gives symmetry of every positive ground state, and the radial
      estimate supplies the far-field control used in the perturbation analysis.

\begin{theorem}[Existence and qualitative properties]\label{thm:existence}
Assume that $N\ge2,$ $1<p<2^*-1$ and $\mu>0.$ 
Then the Nehari level $c_\mu^{\cN}>0$ is attained and
\begin{equation}\label{eq:level-equality-main}
 c_\mu^{\cN}
 =\min_{u\in\mathscr S_\mu}J_\mu(u)
 =\min_{u\in\mathscr S_\mu^{\rad}}J_\mu(u),
\end{equation}
where $\mathscr S_\mu$ and $\mathscr S_\mu^{\rad}$ are defined in Definition~\ref{def:solutions-ground-states}. 
In addition, we have
$\cG_\mu^{\rad}\neq\emptyset.$ Moreover, every positive ground state belongs to $W\cap C^\infty(\R^N)$ and has the form
$u(x)=U(\abs{x})$ after a translation, where
\begin{equation}\label{eq:strict-radial-decay-main}
 U'(r)<0,\qquad
 0<U(r)\le C_N r^{-\frac{N-1}{2}}\norm{u}_2^{\frac{1}{2}}\norm{\nabla u}_2^{\frac{1}{2}}
 \quad(r>0).
\end{equation}  In particular, we also have $U(r)\to0$ as $r\to\infty$.
\end{theorem}

We denote the common level in \eqref{eq:level-equality-main} by \(c_\mu\). 
The next result concerns arbitrary choices from \(\cG_\mu\)
,
before uniqueness is known
For \(\alpha\in(0,1)\), we use the global H\"older norm
\begin{equation}\label{eq:global-holder-norm}
 \norm{v}_{C^{2,\alpha}(\R^N)}
 :=\sum_{j=0}^2\norm{D^jv}_{L^\infty(\R^N)}
   +[D^2v]_{C^{0,\alpha}(\R^N)}.
\end{equation}

Since uniqueness is not yet known, convergence must hold uniformly over
 \(\mathcal{G}_\mu^{\mathrm{rad}}\), without selecting a canonical branch. 
 The Nehari and Pohozaev identities yield uniform bounds for small \(\mu\);
  radial compactness and the classification of positive decaying solutions
   of the limiting equation then identify every subsequential limit with
    \(U_0\). Energy identities give strong convergence in \(W\). Local
     elliptic estimates yield smooth convergence on bounded sets, while
      uniform far-field smallness and the H\"older continuity of
       \(s\log s^2\) near zero give global \(C^{2,\alpha}(\mathbb{R}^N)\) convergence.

\begin{theorem}[Uniform estimates and the Gaussian limit]\label{thm:limit}
Assume that \( N \ge 2\) and \(1 < p < 2^* - 1\).
Let $\mathscr S_\mu^{\rad}$ be defined as in Definition~\ref{def:solutions-ground-states}. 
Then the following statements hold.
Then the following conclusions are true for the sets $\cG_\mu^{\rad}:$
\begin{enumerate}[\rm{(}i\rm{)}]
  \item 

\begin{equation}\label{eq:set-uniform-W-main}
 c_\mu\longrightarrow c_0,
 \qquad
 \lim_{\mu\downarrow0}\ 
 \sup_{u_\mu\in\cG_\mu^{\rad}}\norm{u_\mu-U_0}_W=0.
\end{equation}
\item There exist two constants $\mu_1>0$ and $C=C(N,p,\mu_1)>0,$ such that
\begin{equation}\label{eq:uniform-W-Linfty-main}
 \sup_{0<\mu\le\mu_1}\sup_{u_\mu\in\cG_\mu^{\rad}}
 \bigl(\norm{u_\mu}_W+\norm{u_\mu}_\infty\bigr)\le C
\end{equation}
and for any $\abs{x}\ge1,$
\begin{equation}\label{eq:uniform-algebraic-main}
 \sup_{0<\mu\le\mu_1}\sup_{u_\mu\in\cG_\mu^{\rad}}
 \abs{u_\mu(x)}\le C\abs{x}^{-\frac{N-1}{2}}.
\end{equation}
\item For any given $R>0,$ any integer $k\ge0$ and any $\alpha\in(0,1),$ we have
\begin{equation}\label{eq:local-smooth-main}
 \lim_{\mu\downarrow0}\sup_{u_\mu\in\cG_\mu^{\rad}}
 \norm{u_\mu-U_0}_{C^{k,\alpha}(B_R)}=0.
\end{equation}
Moreover, for each $\alpha\in(0,1),$ we have
\begin{equation}\label{eq:global-smooth-main}
 \lim_{\mu\downarrow0}\sup_{u_\mu\in\cG_\mu^{\rad}}
 \norm{u_\mu-U_0}_{C^{2,\alpha}(\R^N)}=0
\end{equation}
and
\begin{equation}\label{eq:uniform-tail-main}
 \lim_{R\to+\infty}\sup_{0<\mu\le\mu_1}
 \sup_{u_\mu\in\cG_\mu^{\rad}}\sup_{\abs{y}\ge R+2}
 \norm{u_\mu}_{C^{2,\alpha}(B_1(y))}=0.
\end{equation}

\end{enumerate}
\end{theorem}

\begin{remark}
The uniform convergence statements
\eqref{eq:set-uniform-W-main}, \eqref{eq:local-smooth-main} and
\eqref{eq:global-smooth-main} are equivalent to convergence along every
sequence $\mu_n\downarrow0$ and every choice
$u_n\in\cG_{\mu_n}^{\rad}$ in the corresponding norms together with
$c_{\mu_n}\to c_0$. The bounds over $0<\mu\le\mu_1$ in
\eqref{eq:uniform-W-Linfty-main}, \eqref{eq:uniform-algebraic-main}, and
\eqref{eq:uniform-tail-main} are separate assertions.
\end{remark}

For small \(\mu\), we prove uniqueness and nondegeneracy by a perturbative argument
 based on the Gaussian linearization
\[
L_0=-\Delta-\log U_0^2-2=-\Delta+|x|^2-N-2,
\qquad
\ker L_0=\operatorname{span}\{\partial_{x_1}U_0,\ldots,\partial_{x_N}U_0\},
\]
whose radial kernel is trivial.

\begin{theorem}[Uniqueness up to translations and full nondegeneracy]
\label{thm:nondegeneracy}\label{thm:uniqueness-translations}
Assume that $N\ge2$ and $1<p<2^*-1.$ Let
\[
 L_\mu=-\Delta-\log u_\mu^2-2-\mu p\,u_\mu^{p-1}
\]
be the real scalar linearized operator on $L^2(\R^N)$ with the self-adjoint realization
 specified in Definition~\ref{def:linearized-form}. 
Then there exists a constant $\mu_0=\mu_0(N,p)>0,$ such that for every $0<\mu<\mu_0$, there is a unique
$u_\mu\in\cG_\mu^{\rad}$ and
\begin{equation}\label{eq:ground-state-orbit}
 \cG_\mu^+=\{u_\mu(\,\cdot-y):y\in\R^N\}.
\end{equation}
Moreover, we have
\begin{equation*}
 \ker L_\mu
 =\operatorname{span}\{\partial_{x_1}u_\mu,\ldots,
       \partial_{x_N}u_\mu\},
 \qquad
 \ker(L_\mu|_{L^2_{\rad}})=\{0\}.
\end{equation*}
\end{theorem}

In dimensions \(N>2\), a hypothesis of the Serrin--Tang criterion \cite{SerrinTang2000}
 is that \(\frac{f'(s)}{f(s)}\) be nonincreasing where \(f>0\). For \(f_\mu\), let \(q=p-1\)
  and let \(b_\mu\) be the unique positive zero of \(\log s^2+\mu s^q\). For \(s>b_\mu\),
\[
\frac{d}{ds}\left(\frac{sf_\mu'(s)}{f_\mu(s)}\right)
=\frac{2\mu q s^q(q\log s-2)-4}
{s\bigl(\log s^2+\mu s^q\bigr)^2},
\]
which is positive for all sufficiently large \(s\), for every \(\mu>0\). 
Thus this criterion does not apply directly, even for arbitrarily small \(\mu\).

Theorem~\ref{thm:limit} gives local convergence of the coefficients
 in the normalized-difference and linearized equations. These coefficients
  may be unbounded below, and local convergence alone cannot prevent \(L^2\)-normalized
   differences or kernel functions from vanishing on every bounded set. We obtain
    a uniform global upper bound for the coefficients and use uniform far-field
     smallness to make them uniformly negative outside a fixed ball. A cutoff
      argument justifies the global energy identity, which then forces a fixed
       positive amount of \(L^2\)-mass to remain in that ball. Compactness 
       yields a nonzero limit in \(\ker L_0\).

For differences of radial ground states, this contradicts the trivial radial kernel;
 combined with Theorem~\ref{thm:existence}, the resulting
  radial uniqueness gives uniqueness up to translations. For kernel
   functions chosen orthogonal to the translation modes, strong \(H^1\) 
   convergence of the ground states preserves this orthogonality in the limit,
    excluding additional kernel modes. The global H\"older convergence in
     Theorem~\ref{thm:limit} is an additional regularity conclusion; the 
     perturbative argument uses energy-space convergence, local convergence, 
     and uniform far-field smallness, without requiring precise Gaussian decay of
      the perturbed states.

The paper is organized as follows. Section~\ref{sec:Preliminaries} introduces 
the functional setting and preliminary estimates. Section~\ref{sec:existence}
 proves existence and the qualitative properties of ground states. Section~\ref{sec:limit}
  establishes uniform bounds and convergence to the Gaussian. 
  Section~\ref{sec:uniqueness-nondegeneracy} proves uniqueness and nondegeneracy.
   The appendix contains the iteration and H\"older estimates used in the global
    regularity analysis.

Throughout the paper, functions are real-valued. We write $\norm{\cdot}_q$ for the norm in $L^q(\R^N)$, $B_R(y)$ for the open ball of radius $R$ centered at $y$, and $B_R=B_R(0)$.
  The symbols $t_+=\max\{t,0\}$ and $t_-=\max\{-t,0\}$ denote the positive and negative parts. Constants denoted by $C$ may change from line to line; their dependence is specified when uniformity is needed.

\section{Preliminaries}\label{sec:Preliminaries}

\subsection{Functional setting and variational definitions}\label{sec:variational}
We first introduce the logarithmic energy space and recall some estimates used in the sequel.

As in \cite{Cazenave1983}, for any $s\ge0,$ define
\begin{equation*}
 A(s)=\begin{cases}
 -s^2\log s^2,&0\le s\le e^{-3},\\
 3s^2+4e^{-3}s-e^{-6},&s\ge e^{-3},
 \end{cases}
 \qquad
 B(s)=s^2\log s^2+A(s),
\end{equation*}
We adopt the convention $0\log 0=0$ and extend both functions evenly to $\R$.
The two expressions for $A$, as well as their first derivatives, agree at
$s=e^{-3}$. Moreover, a direct computation shows that $A$ is convex on
$[0,\infty)$. Hence, after the even extension, $A$ is an even convex
$C^1$ Young function.

\begin{definition}[Logarithmic energy space]\label{def:energy-space}
Following \cite{Cazenave1983}, let
\[
 V=L^A(\R^N)=\{u\in L^1_{\loc}(\R^N):A(\abs u)\in L^1(\R^N)\},
\]
equipped with the Luxemburg norm
\[
 \norm{u}_V=\inf\left\{k>0:
 \int_{\R^N}A\!\left(\frac{\abs u}{k}\right)\dd\le1\right\}.
\]
The logarithmic Orlicz-Sobolev space and its radial subspace are denoted by
\begin{align*}
 W &= H^1(\R^N)\cap V,
 \qquad
 \norm{u}_W=\norm{u}_{H^1}+\norm{u}_V,\\
 W_{\rad}
 &=\{u\in W : u\circ Q=u \text{ for every } Q\in O(N)\}.
\end{align*}
\end{definition}

 Indeed, for any $r\in(2,2^*)$, we have 
\begin{equation}\label{eq:positive-log-growth}
 (s^2\log s^2)^+\le C_r|s|^r, 
\end{equation}
     together with the Sobolev embedding, we can obtain the following conclusion,
For every $u\in H^1(\mathbb R^N)$, the positive part of $u^2\log u^2$
is integrable.
Furthermore, the definition of $A$ then gives the equivalence
\begin{equation}\label{eq:W-log-equivalence}
 u\in W
 \quad\Longleftrightarrow\quad
 u\in H^1(\R^N)\ \text{and}\ 
 \int_{\R^N}\abs{u}^2\abs{\log u^2}\dd<+\infty.
\end{equation}
In particular, both the positive and negative parts of the logarithmic integral are finite.

Accordingly, the energy associated with~\eqref{eq:action-intro} is well defined on $W$ and can be written as
\begin{equation*}
 J_\mu(u)=\frac12\int\bigl(\abs{\nabla u}^2+u^2\bigr)
 +\frac12\int A(\abs{u})-\frac12\int B(\abs{u})
 -\frac{\mu}{p+1}\int\abs{u}^{p+1}.
\end{equation*}
The Orlicz estimates below give reflexivity and a criterion for norm convergence from modular convergence. These will be used in the
variational construction and in the Gaussian limit.
 
\begin{proposition}[Orlicz estimates]\textnormal{(see \cite{RaoRen1991} and \cite[Lemma~2.1]{Cazenave1983})}\label{prop:orlicz-facts}
Write $\rho_A(v)=\int A(\abs v)$ and let $\widetilde A$ be the
complementary Young function of $A$. 
\begin{enumerate}[\rm{(}i\rm{)}]
  \item 
For $s>0$, we have
\begin{equation}\label{eq:A-growth-indices}
 1\le\frac{sA'(s)}{A(s)}\le2.
\end{equation}
\item 
Both $A$ and $\widetilde A$ satisfy a global $\Delta_2$ condition,
then $V$ and $W$ are reflexive, and
\begin{equation}\label{eq:modular-norm}
 \min\{\norm v_V,\norm v_V^2\}
 \le\rho_A(v)\le
 \max\{\norm v_V,\norm v_V^2\}.
\end{equation}
\item 
If $v_n\to v$ almost everywhere and
$\rho_A(v_n)\to\rho_A(v)<\infty$, then $v_n\to v$ in $V$.
\item 
Moreover, $C_c^\infty(\R^N)$ is dense in $W$.
\end{enumerate}
\end{proposition}

\begin{proposition}[Nonlinear integral calculus]
  \label{prop:nonlinear-calculus}\leavevmode
  \begin{enumerate}[\rm{(}i\rm{)}]
 \item
For every fixed $r>2$ and all real $s,t$,
\begin{equation}\label{eq:B-mean-value}
 \begin{split}
 |B'(s)|&\le C_r|s|^{r-1},\\
 |B(s)-B(t)|&\le C_r(|s|^{r-1}+|t|^{r-1})|s-t|.
 \end{split}
\end{equation}
\item 
In particular,
\[
    v_n\to v \text{ in }L^r(\mathbb R^N)
    \quad\Longrightarrow\quad
    B(v_n)\to B(v) \text{ in }L^1(\mathbb R^N).
\]

\item 
The functionals $J_\mu$ and $K_\mu$ are of class $C^1$ on $W$ for every
$\mu\ge0$, and

\begin{equation}\label{eq:action-derivative}
 J_\mu'(u)[\varphi]
 =\int_{\R^N}\nabla u\cdot\nabla\varphi
 -\int_{\R^N} u\log u^2\,\varphi
 -\mu\int_{\R^N}\abs{u}^{p-1}u\,\varphi.
\end{equation}
\item 
Consequently, every distributional solution in $W$ is a weak solution in the sense of
Definition~\ref{def:solutions-ground-states}.
\end{enumerate}
\end{proposition}

\begin{proof}
Since $B$ vanishes in a neighborhood of the origin and
\[
    B'(s)=O\bigl(|s|(1+\log^+|s|)\bigr)
    \qquad\text{as }|s|\to\infty,
\]
for every fixed $r>2$ we have
\[
    |B'(s)|\le C_r |s|^{r-1}.
\]
The second estimate in \eqref{eq:B-mean-value} then follows from the
mean-value theorem, and assertion {\rm (ii)} follows immediately from
H\"older's inequality.

The $C^1$-regularity of the logarithmic energy functional on the
corresponding Orlicz--Sobolev space is standard; see
\cite[Propositions~2.7 and~3.1]{Cazenave1983}
and \cite[Lemma~2.2]{ShuaiYang2025}.
Hence $J_\mu\in C^1(W,\mathbb R)$ and its derivative is given by
\eqref{eq:action-derivative}. Moreover, using
\[
    K_\mu(u)
    =
    \|\nabla u\|_2^2
    +
    \int_{\mathbb R^N}A(|u|)\,dx
    -
    \int_{\mathbb R^N}B(|u|)\,dx
    -
    \mu\|u\|_{p+1}^{p+1}
\]
and  the subcritical Sobolev embedding,
we also obtain $K_\mu\in C^1(W,\mathbb R)$.

Finally, since $C_c^\infty(\mathbb R^N)$ is dense in $W$ and
$J_\mu'(u)\in W^*$, the distributional identity extends by continuity
from $C_c^\infty(\mathbb R^N)$ to all test functions
$\varphi\in W$. This proves {\rm(iv)}.
\end{proof}

\begin{definition}[Solutions and ground states]\label{def:solutions-ground-states}\leavevmode

  \begin{enumerate}
  \item []
A \emph{weak solution} is $u\in W$ such that $J_\mu'(u)=0$ in $W^*$, 
or, equivalently, such that \eqref{eq:action-derivative} vanishes for every
$\varphi\in W$.
\item[]
A \emph{centered radial positive solution} is positive and has
the form $u(x)=U(\abs{x})$ with its center fixed at the origin.
  \end{enumerate}

Set
\begin{equation*}
 \mathscr S_\mu=\{u\in W\setminus\{0\}:J_\mu'(u)=0\},
 \qquad
 \mathscr S_\mu^{\rad}=\mathscr S_\mu\cap W_{\rad}.
\end{equation*}
A \emph{ground state} minimizes $J_\mu$ over $\mathscr S_\mu$. Similaryly, a \emph{radial ground state} minimizes $J_\mu$ over $\mathscr S_\mu^{\rad}$. Theorem~\ref{thm:existence} shows that these two minimum levels agree.
\end{definition}

\begin{definition}[Nehari manifold and levels]\label{def:nehari}
Define
\begin{equation}\label{eq:K-definition}
 K_\mu(u)=J_\mu'(u)[u]
 =\norm{\nabla u}_2^2-\int u^2\log u^2
   -\mu\norm{u}_{p+1}^{p+1},
\end{equation}
\begin{equation*}
 \cN_\mu=\{u\in W\setminus\{0\}:K_\mu(u)=0\},
 \qquad c_\mu^{\cN}=\inf_{u\in\cN_\mu}J_\mu(u).
\end{equation*}
Once the equality \eqref{eq:level-equality-main} has been established, we denote its common value by
$c_\mu$ and write $\mathcal G_\mu$ for the set of ground states. We also set
\begin{equation*}
 \cG_\mu^+=\{u\in\cG_\mu:u>0\},\qquad
 \cG_\mu^{\rad}=\cG_\mu^+\cap W_{\rad}.
\end{equation*}
\end{definition}

\begin{definition}[Nondegeneracy]\label{def:nondegeneracy}
For a centered radial solution with a specified realization of its linearized operator, 
\emph{radial nondegeneracy} means that its kernel space in the
radial subspace is $\{0\}$.

\emph{Full nondegeneracy} means that its kernel space in $L^2(\mathbb R^N)$ is precisely
the span of the translation derivatives.

\end{definition}

\begin{proposition}[Radial compactness]\label{prop:radial-compactness}\textnormal{\cite[Proposition~3.1]{Cazenave1983}}
  If $u_n\rightharpoonup u$ in $W_{\mathrm{rad}}$, then
\begin{equation}\label{eq:radial-compactness}
    u_n\to u \quad\text{strongly in }L^q(\mathbb R^N)
    \quad\text{for every }2\le q<2^*.
\end{equation}

\end{proposition}

\subsection{Auxiliary estimates}\label{sec:preliminaries}

Throughout the paper, $C>0$ denotes a generic constant that may vary 
from line to line and depends only on $N$, $p$, the fixed exponents, 
and any prescribed upper bound such as $\mu_1$, unless otherwise stated.
 In particular, $C$ is independent of the particular solution and the sequence index.

We first record the estimates used below.  The Euclidean logarithmic
Sobolev inequality in the normalization of \cite{LiebLoss2001} states
that, for $u\in H^1(\R^N)\setminus\{0\}$ and $a>0$,
\begin{equation}\label{eq:log-sobolev}
 \int u^2\log u^2
 \le \frac{a^2}{\pi}\norm{\nabla u}_2^2+
 \bigl(\log\norm{u}_2^2-N(1+\log a)\bigr)\norm{u}_2^2.
\end{equation}
In every application below, we choose $a$ so that $\frac{a^2}{\pi}<1$ and verify separately that the $L^2$ norm is bounded.

\begin{lemma}[Pointwise logarithmic estimates]\label{lem:pointwise-log}\leavevmode
  \begin{enumerate}[\rm{(}i\rm{)}]
        \item 
For every $r>2$ and $\kappa>0$, there exists a constant
$C=C(r,\kappa)>0$ such that
\begin{equation}
    s^2\log s^2
    \le -\kappa s^2+C|s|^r,
    \qquad s\in\mathbb R.
    \label{eq:negative-log-control}
\end{equation}
\item 
For every $\varepsilon\in(0,1)$, there exists
$C_\varepsilon>0$ such that
\begin{equation*}
    |s\log s^2|
    \le C_\varepsilon
    \bigl(
        |s|^{1-\varepsilon}
        +
        |s|^{1+\varepsilon}
    \bigr),
    \qquad s\in\mathbb R.
\end{equation*}
\item 
More generally, for every $\delta\in(0,1)$ and $r>2$, there exists
$C_{\delta,r}>0$ such that
\begin{equation}
    |s\log s^2|
    \le C_{\delta,r}
    \bigl(
        |s|^{1-\delta}
        +
        |s|^{r-1}
    \bigr),
    \qquad s\in\mathbb R.
    \label{eq:log-vitali-growth}
\end{equation}
\end{enumerate}
\end{lemma}

\begin{proof}
For \eqref{eq:negative-log-control}, the positive part of
$(\frac{\log s^2+\kappa}{|s|^{r-2}})$ vanishes for sufficiently small
$|s|$ and tends to zero at infinity, so it is bounded. The remaining estimates follow by splitting at
$|s|=1$ and using the boundedness of
$t^\delta|\log t|$ on $(0,1]$ and
$t^{-(r-2)}\log t$ on $[1,\infty)$.
Related logarithmic growth estimates appear in
\cite{AlvesMoussaouiTavares2019}.
\end{proof}

\begin{lemma}[Radial $H^1$ bound]\label{lem:radial-bound}
  Let $u\in H^1_{\mathrm{rad}}(\mathbb R^N)$ and write
$u(x)=U(|x|)$, where $U\in AC_{\mathrm{loc}}(0,\infty)$ is its radial
representative.
Then, for every $r>0$,
\begin{equation*}
 \abs{U(r)}^2\le C_N r^{1-N}\norm{u}_2\norm{\nabla u}_2,
\end{equation*}
and hence $U(r)\to0$ as $r\to\infty$.  
\end{lemma}

\begin{proof}
This is the classical Strauss radial lemma; see
\cite{Strauss1977}.
\end{proof}
For the limiting equation, testing with the solution itself requires control of both the positive and negative parts of the logarithmic integral.

\begin{lemma}[Testing a distributional logarithmic equation by the solution]
\label{lem:test-by-solution}
Suppose $u\in W$ satisfies
\begin{equation*}
 \int_{\R^N}\nabla u\cdot\nabla\varphi
 =\int_{\R^N} u\log u^2\,\varphi
\end{equation*}
for every $\varphi\in C_c^\infty(\R^N)$.  Then we have
\begin{equation}\label{eq:limit-nehari-identity}
 \norm{\nabla u}_2^2=\int_{\R^N} u^2\log u^2.
\end{equation}
\end{lemma}

\begin{proof}
By Proposition~\ref{prop:nonlinear-calculus} and formula~\eqref{eq:action-derivative}, the hypothesis is equivalent to
$J_0'(u)[\varphi]=0$ for every $\varphi\in C_c^\infty(\mathbb R^N)$.
Since $J_0'(u)\in W^*$ and $C_c^\infty(\mathbb R^N)$ is dense in $W$ by
Proposition~\ref{prop:orlicz-facts}, it follows that $J_0'(u)=0$ on $W$. Taking
$\varphi=u$ then yields \eqref{eq:limit-nehari-identity}.
\end{proof}

For later cutoff arguments, fix $\eta\in C_c^\infty(\R^N)$ with $0\le\eta\le1$, $\eta=1$ on $B_1$, and $\eta=0$ outside $B_2$, and set $\eta_R(x)=\eta(\frac x R)$.
 Then $\abs{\nabla\eta_R}\le \frac C R$.

The local compactness, weak lower semicontinuity, Sobolev embedding, and
Hilbert-space norm arguments used below can be found in \cite{Brezis2011}.  The
interior $W^{2,q}$ and Schauder estimates are used in their
translation-invariant ball versions from \cite{GilbargTrudinger1977}.

\section{Existence and qualitative properties}\label{sec:existence}

We prove Theorem~\ref{thm:existence} by minimizing on the Nehari
constraint. Rearrangement allows us to use radial compactness, while rescaling restores the constraint at the weak limit. After obtaining
a minimizer, we verify that the constraint is natural and establish
positivity and strict radial decrease. 
The equality case of rearrangement then establishes the symmetry of every positive ground state.

\subsection{Minimization on the Nehari manifold}

Schwarz rearrangement decreases the Dirichlet integral. The following scaling property restores the Nehari constraint and controls the action.

\begin{lemma}[Nehari projection]\label{lem:nehari-projection}

For every $u\in W\setminus\{0\}$, there exists a unique $t(u)>0$ such that
\[
K_\mu(t(u)u)=0,
\]
or equivalently,
\[
t(u)u\in\mathcal N_\mu.
\]
Moreover, we obtain
\[
K_\mu(u)<0 \ \Longrightarrow\ 0<t(u)<1,
\qquad
K_\mu(u)=0 \ \Longrightarrow\ t(u)=1.
\]
In particular, if $K_\mu(u)\le 0$, then $t(u)\le 1$.
Furthermore, for every $u\in\mathcal N_\mu$,
\begin{equation}\label{eq:energy-on-nehari}
 J_\mu(u)=\frac12\norm{u}_2^2
 +\frac{\mu(p-1)}{2(p+1)}\norm{u}_{p+1}^{p+1}>0.
\end{equation}

\end{lemma}

\begin{proof}
Put $G=\|\nabla u\|_2^2$, $M=\|u\|_2^2>0$, $H=\int u^2\log u^2$, and $P=\|u\|_{p+1}^{p+1}$. For $t>0$,
\begin{equation}\label{eq:fiber-K}
F_u(t):=\frac{K_\mu(tu)}{t^2}=G-H-M\log t^2-\mu t^{p-1}P.
\end{equation}
This function tends to $+\infty$ as $t\downarrow0$ and to $-\infty$ as $t\to\infty$, while
\[
F_u'(t)=-\frac{2M}{t}-\mu(p-1)t^{p-2}P<0.
\]
It therefore has exactly one zero $t(u)$. Since $F_u(1)=K_\mu(u)$, strict monotonicity gives the stated comparison with one. Finally, for $v\in\cN_\mu$, the identity $K_\mu(v)=0$ gives
$\int v^2\log v^2=\|\nabla v\|_2^2-\mu\|v\|_{p+1}^{p+1}$.
Substitution into $J_\mu(v)$ yields \eqref{eq:energy-on-nehari}, whose right-hand side is positive because $v\ne0$.
\end{proof}

\begin{lemma}[Separation and positivity of the level]\label{lem:positive-level}
There exists $\rho=\rho(N,p,\mu)>0$ such that
\begin{equation}\label{eq:nehari-separation}
 \norm{u}_{2}^2+\norm{\nabla u}_{2}^2\ge\rho^2\qquad(u\in\cN_\mu),
\end{equation}
and $0<c_\mu^{\cN}<\infty$.
\end{lemma}

\begin{proof}
Choose $r\in(2,2^*)$ and $\kappa>0$.  From
\eqref{eq:negative-log-control}, $K_\mu(u)=0$ and Sobolev embedding,
\begin{equation*}
  \min\{1,\kappa\}\norm{u}_{H^1}^2
 \le \norm{\nabla u}_2^2+\kappa\norm{u}_2^2
 \le C_{\kappa,r}\norm{u}_r^r+\mu\norm{u}_{p+1}^{p+1}
 \le C\norm{u}_{H^1}^r+C\norm{u}_{H^1}^{p+1}.
\end{equation*}
Division by $\norm{u}_{H^1}^2$ proves
\eqref{eq:nehari-separation}.
For any nonzero $\phi\in C_c^\infty(\R^N)$,
Lemma~\ref{lem:nehari-projection} gives
$c_\mu^{\cN}\le J_\mu(t(\phi)\phi)<\infty$.
Equation \eqref{eq:energy-on-nehari} gives nonnegativity.

If $c_\mu^{\cN}=0$, take $u_n\in\cN_\mu$ with $J_\mu(u_n)\to0$.
Equation \eqref{eq:energy-on-nehari} yields
$\norm{u_n}_2\to0$ and $\norm{u_n}_{p+1}\to0$.  Fix $a>0$ with
$\eta=\frac{a^2}{\pi}<1$.  Combining \eqref{eq:log-sobolev} with
$K_\mu(u_n)=0$ gives
\begin{equation}\label{eq:positive-level-calculation}
 (1-\eta)\norm{\nabla u_n}_2^2
 \le \norm{u_n}_2^2\log\norm{u_n}_2^2
 -N(1+\log a)\norm{u_n}_2^2
 +\mu\norm{u_n}_{p+1}^{p+1}\rightarrow 0\quad (n\rightarrow\infty).
\end{equation}
Therefore, $u_n\to0$ in $H^1$, contrary to
\eqref{eq:nehari-separation}.
\end{proof}

\begin{proposition}[Attainment]\label{prop:attainment}
The level $c_\mu^{\cN}$ is attained by a nonnegative, radial, radially nonincreasing function $u\in W$.
\end{proposition}

\begin{proof}
Let $v_n\in\cN_\mu$ be a minimizing sequence.  Replace $\abs{v_n}$ by its Schwarz
rearrangement $w_n=\abs{v_n}^*$.  Equimeasurability preserves the
$L^2$, $L^{p+1}$, $A$, $B$, and logarithmic integrals, while the
P\'olya--Szeg\H{o} inequality gives
$\norm{\nabla w_n}_2\le\norm{\nabla v_n}_2$; these rearrangement properties are stated in \cite[Chapter~3]{LiebLoss2001}.  Consequently
$K_\mu(w_n)\le0$.  Let $t_n=t(w_n)\le1$ and set $u_n=t_nw_n$.  Then
$u_n\in\cN_\mu$, and \eqref{eq:energy-on-nehari} shows
\[
 c_\mu^{\cN}\le J_\mu(u_n)\le J_\mu(v_n).
\]
Thus $(u_n)$ is a nonnegative, radial, radially nonincreasing minimizing sequence.

Identity \eqref{eq:energy-on-nehari} yields uniform bounds for
\(\|u_n\|_2\) and \(\|u_n\|_{p+1}\). Choosing \(a>0\) so that \(\frac{a^2}{\pi}<1\),
 the argument leading to \eqref{eq:positive-level-calculation}
also yields a uniform bound for \(\|\nabla u_n\|_2\).
  From $u_n^2\log u_n^2=B(u_n)-A(u_n)$ and the
Nehari identity,
\begin{equation}\label{eq:A-modular-bound}
 \int_{\R^N} A(\abs{u_n})
 =\int_{\R^N} B(\abs{u_n})-\norm{\nabla u_n}_2^2
   +\mu\norm{u_n}_{p+1}^{p+1}
 \le C\norm{u_n}_r^r+\mu\norm{u_n}_{p+1}^{p+1}
\end{equation}
for a fixed $r\in(2,2^*)$.  Proposition~\ref{prop:orlicz-facts} converts
this modular bound into a $V$-norm bound.  Hence $(u_n)$ is bounded in
$W_{\rad}$.

After passing to a subsequence, \(u_n\rightharpoonup u\) in \(W\).
By Proposition~\ref{prop:radial-compactness}, we have
\[
u_n\to u \quad\text{strongly in }L^2(\mathbb R^N)
\text{ and in }L^r(\mathbb R^N)\ \text{for every }2<r<2^*,
\]
and, after passing to a further subsequence if necessary,
\[
u_n(x)\to u(x)\qquad\text{for a.e. }x\in\mathbb R^N.
\]
Hence, by \eqref{eq:B-mean-value}, \(B(u_n)\to B(u)\) in \(L^1(\mathbb R^N)\).

Passing to the limit in \eqref{eq:energy-on-nehari} yields
\begin{equation}\label{eq:limit-energy-minimizer}
 c_\mu^{\cN}
 =\frac12\norm{u}_2^2
 +\frac{\mu(p-1)}{2(p+1)}\norm{u}_{p+1}^{p+1},
\end{equation}
Since \(c_\mu^\cN>0\) by Lemma~\ref{lem:positive-level}, it follows that \(u\neq0\).

Weak lower semicontinuity of the convex functional
$\norm{\nabla v}_2^2+\int A(\abs v)$, together with convergence of the
$B$ and power terms, gives $K_\mu(u)\le0$.  Let $t=t(u)\le1$.  Since
$tu\in\cN_\mu$, \eqref{eq:limit-energy-minimizer} gives
\[
 c_\mu^{\cN}\le J_\mu(tu)
 =\frac{t^2}{2}\norm{u}_2^2
 +\frac{\mu(p-1)t^{p+1}}{2(p+1)}\norm{u}_{p+1}^{p+1}
 \le c_\mu^{\cN}.
\]
Because $\norm{u}_2>0$, equality forces $t=1$.  Thus
$u\in\cN_\mu$ and $J_\mu(u)=c_\mu^{\cN}$.

The almost-everywhere limit of the nonincreasing radial representatives
is nonnegative and nonincreasing.
\end{proof}

\begin{proposition}[Natural constraint]\label{prop:natural-constraint}
Every minimizer of $J_\mu$ on $\cN_\mu$ is an unconstrained critical
point of $J_\mu$.
\end{proposition}

\begin{proof}
For $u\in\cN_\mu$, differentiation of $K_\mu(tu)$ at $t=1$ gives
\begin{equation}\label{eq:K-radial-derivative}
 K_\mu'(u)[u]
 =-2\norm{u}_2^2-\mu(p-1)\norm{u}_{p+1}^{p+1}<0.
\end{equation}
Thus $\cN_\mu$ is locally a $C^1$ hypersurface. The Lagrange
multiplier theorem gives $J_\mu'(u)=\lambda K_\mu'(u)$ at a
constrained minimizer.

Evaluating the Lagrange multiplier identity on \(u\) gives
\[
0=J_\mu'(u)[u]
 =\lambda K_\mu'(u)[u].
\]
By \eqref{eq:K-radial-derivative}, \(K_\mu'(u)[u]<0\); hence \(\lambda=0\), and therefore
\(J_\mu'(u)=0\).

\end{proof}

\subsection{Regularity, positivity, and variational identities}

The logarithmic nonlinearity is smooth on $(0,\infty)$ but has an unbounded derivative at zero. We first establish boundedness and strict positivity, and then apply elliptic regularity.

\begin{proposition}[Regularity and positivity]\label{prop:regularity}

For every \(\mu\ge0\). Every nonzero, nonnegative distributional solution
\(u\in W\) of \eqref{eq:main} satisfies
\[
u\in C^\infty(\mathbb R^N),
\qquad
u>0 \quad\text{in }\mathbb R^N.
\]
Moreover, its distributional formulation extends to all test functions
\(\varphi\in W\).

\end{proposition}

\begin{proof}
  By Proposition~\ref{prop:nonlinear-calculus}, the distributional identity extends to every test function
$\varphi\in W$.
 Lemma~\ref{lem:moser-iteration}, applied with $\bar\mu=\max\{1,\mu\}$ and $M_H=\norm u_{H^1}$,
  gives $u\in L^\infty(\R^N)$. 
  
  The map \(s\mapsto s\log s^2+\mu s^p\), extended continuously at \(s=0\),
is bounded on the range of \(u\). Hence the interior \(W^{2,q}\)-estimates give
\begin{equation}
u\in W_{\mathrm{loc}}^{2,q}(\mathbb R^N)
\qquad\text{for every }1<q<\infty.
\end{equation}
Choosing \(q>N\) sufficiently large and using the Sobolev embedding then yields
\begin{equation}\label{eq:preliminary-regularity}
u\in C_{\mathrm{loc}}^{1,\alpha}(\mathbb R^N)
\qquad\text{for every }0<\alpha<1.
\end{equation}

It remains to prove strict positivity without assuming local Lipschitz continuity at zero.  Set, for $s\ge0$,
\begin{equation*}
 \beta(s)=-s\log s^2-\mu s^p,\qquad \beta(0)=0.
\end{equation*}
Then $-\Delta u+\beta(u)=0$, and
\[
 \beta'(s)=-\log s^2-2-\mu p s^{p-1}\longrightarrow+\infty
 \quad(s\downarrow0).
\]
Choose $\delta>0$ small enough that $\beta$ is nonnegative and nondecreasing on $[0,\delta]$.
Moreover,
\begin{equation*}
 \beta(s)\sim2s\abs{\log s},
 \qquad
 \int_0^s\beta(t)\,dt\sim s^2\abs{\log s}
 \quad(s\downarrow0),
\end{equation*}
and consequently
\begin{equation}\label{eq:osgood-divergence}
 \int_0^\delta\frac{ds}{\sqrt{s\beta(s)}}=\infty,
 \qquad
 \int_0^\delta\frac{ds}
 {\sqrt{\int_0^s\beta(t)\,dt}}=\infty.
\end{equation}
Indeed, each integrand is comparable to
$\frac 1 {s\sqrt{\abs{\log s}}}$; substituting
$y=-\log s$ proves divergence.  Extend $\beta|_{[0,\delta]}$
continuously and nondecreasingly to $[0,\infty)$ by keeping it constant after $\delta$.  If $u(x_0)=0$, continuity provides a ball on
which $0\le u<\delta$.  The localized strong maximum principle
\cite[Theorem~1]{Vazquez1984} applies there because
\eqref{eq:osgood-divergence} verifies its Osgood condition, and
$u,\Delta u\in L^1_{\loc}$ by \eqref{eq:preliminary-regularity}.
It follows that $u$ vanishes throughout that ball.  Hence the zero set is
both open and closed.  Connectedness and $u\not\equiv0$ imply $u>0$.

For every compact set \(K\Subset\mathbb R^N\), positivity and continuity give
\[
0<m_K\le u(x)\le M_K<\infty
\qquad (x\in K).
\]
Hence \(s\mapsto s\log s^2+\mu s^p\) is smooth on
\([m_K,M_K]\Subset(0,\infty)\). Schauder bootstrapping, using the interior
estimates of \cite{GilbargTrudinger1977}, therefore gives
\(u\in C^\infty(\mathbb R^N)\).

\end{proof}

\begin{lemma}[Pohozaev identities]\label{lem:pohozaev}
Let $u\in W$ be a positive weak solution with parameter $\mu\ge0$.
Then
\begin{align}
 \norm{\nabla u}_2^2
 &=\int_{\R^N} u^2\log u^2+\mu\norm{u}_{p+1}^{p+1},
 \label{eq:nehari-solution}\\
 \frac{N-2}{2}\norm{\nabla u}_2^2
 &=\frac N2\int_{\R^N} u^2\log u^2-\frac N2\norm{u}_2^2
   +\frac{N\mu}{p+1}\norm{u}_{p+1}^{p+1}.
 \label{eq:pohozaev}
\end{align}
Consequently,
\begin{equation}\label{eq:pohozaev-ground-identities}
 \norm{\nabla u}_2^2=NJ_\mu(u),\qquad
 J_\mu(u)=\frac12\norm{u}_2^2
 +\frac{\mu(p-1)}{2(p+1)}\norm{u}_{p+1}^{p+1}.
\end{equation}
\end{lemma}

\begin{proof}
Set
\[
 F_\mu(s)=\tfrac12s^2\log s^2-\tfrac12s^2
                 +\tfrac{\mu}{p+1}|s|^{p+1},\qquad F_\mu(0)=0.
\]
The definition of $W$ and the subcritical Sobolev embedding imply
$F_\mu(u)\in L^1$. Proposition~\ref{prop:regularity} gives the local smoothness needed for compactly supported integration by parts.
The first identity follows by testing the weak formulation with $u$. For the second, multiply the equation by
$\eta_R\,x\cdot\nabla u$, with the cutoff defined in Section~\ref{sec:preliminaries}. Integration by parts gives
\[
 \frac{N-2}{2}\int_{\R^N}\eta_R|\nabla u|^2-N\int_{\R^N}\eta_RF_\mu(u)=E_R,
\]
where
\[
 E_R=\int_{\R^N}(x\cdot\nabla\eta_R)F_\mu(u)
      -\frac12\int_{\R^N}(x\cdot\nabla\eta_R)|\nabla u|^2
      +\int_{\R^N}(\nabla\eta_R\cdot\nabla u)(x\cdot\nabla u).
\]
Since $|x||\nabla\eta_R|\le C$ on the annulus supporting the cutoff gradient,
\begin{equation*}
 |E_R|\le C\int_{R<|x|<2R}
                  (|\nabla u|^2+|F_\mu(u)|)\longrightarrow0.
\end{equation*}
Letting \(R\to\infty\), using \(E_R\to0\) and dominated convergence in the two
main integrals, yields \eqref{eq:pohozaev}.
Eliminating the logarithmic integral from \eqref{eq:nehari-solution} and  \eqref{eq:pohozaev} yields
\[
\|\nabla u\|_2^2
 =\frac N2\|u\|_2^2
 +\frac{N\mu(p-1)}{2(p+1)}\|u\|_{p+1}^{p+1}
 =NJ_\mu(u),
\]
which proves \eqref{eq:pohozaev-ground-identities}.
For $N=2$ the Dirichlet term in \eqref{eq:pohozaev} vanishes, but
the same elimination remains valid and gives $\norm{\nabla u}_2^2=2J_\mu(u)$.
No decay estimate beyond the defining integrability of $W$ is used.
\end{proof}

\subsection{Strict monotonicity and symmetry}

The radial equation gives strict monotonicity of the rearranged minimizer. Equality in the P\'olya--Szeg\H{o} inequality then yields symmetry of every positive ground state.

\begin{proposition}[Strict decrease of nonincreasing radial solutions]
\label{prop:strict-decrease}
For every nonzero, nonnegative, radial, radially nonincreasing weak
solution $u\in W$ , there exists a smooth radial profile $U$ satisfying
$U'(r)<0$ for $r>0$ and $\lim_{r\to\infty}U(r)=0$.
In particular, this applies to the
minimizer in Proposition~\ref{prop:attainment}.
\end{proposition}

\begin{proof}
By Proposition~\ref{prop:regularity}, $u$ is smooth and positive. 
Its radial profile \(U\) is positive and nonincreasing. Since a nonnegative,
nonincreasing radial profile belonging to \(L^2(\mathbb R^N)\) must satisfy
\[
U(r)\to0\qquad\text{as }r\to\infty,
\]
we have the asserted limit.
 The radial equation is
\begin{equation}\label{eq:radial-ode}
 -U''-\frac{N-1}{r}U'=U\log U^2+\mu U^p.
\end{equation}
Set $Z=U'\le0$. Differentiating gives
\[
 -Z''-\frac{N-1}{r}Z'+\frac{N-1}{r^2}Z
 =(\log U^2+2+\mu pU^{p-1})Z\qquad(r>0).
\]
If $Z(r_0)=0$ for some $r_0>0$, then $Z'(r_0)=0$ because $r_0$ is a local maximum of $Z$. Uniqueness for this linear initial-value problem, whose coefficients are continuous on $(0,\infty)$, gives $Z\equiv0$. This would make $U$ a positive constant, contradicting $u\in L^2(\R^N)$. Thus $U'(r)<0$ for every $r>0$.
\end{proof}

\begin{proposition}[Symmetry of nonnegative Nehari minimizers]
\label{prop:positive-symmetry}
Let $N\ge2$, $1<p<2^*-1$, and $\mu>0$. If $u\ge0$ belongs to
$\cN_\mu$ and $J_\mu(u)=c_\mu^{\cN}$, there exists $x_0\in\R^N$
and a radial profile $U$ such that
\begin{equation}\label{eq:ground-state-symmetry}
 u(x)=u^*(x-x_0)=U(\abs{x-x_0}),\qquad
 U'(r)<0\quad(r>0),\qquad \lim_{r\to\infty}U(r)=0.
\end{equation}
\end{proposition}

\begin{proof}
Set $w=u^*$. Propositions~\ref{prop:natural-constraint} and
\ref{prop:regularity} first show that $u$ is smooth and positive. Equimeasurability preserves the $L^2$ and $L^{p+1}$
norms and the two nonnegative integrals of
$(u^2\log u^2)^+$ and $(u^2\log u^2)^-$, which are finite by
\eqref{eq:W-log-equivalence}. The P\'olya--Szeg\H{o} inequality
\cite[Chapter~3]{LiebLoss2001} gives $w\in H^1$ and
$\norm{\nabla w}_2\le\norm{\nabla u}_2$. Hence $w\in W$,
$\int w^2\log w^2=\int u^2\log u^2$, and
\begin{equation*}
 K_\mu(w)=\norm{\nabla w}_2^2-\norm{\nabla u}_2^2\le0.
\end{equation*}
By Lemma~\ref{lem:nehari-projection}, there is $t\in(0,1]$ such that
$tw\in\cN_\mu$. Applying \eqref{eq:energy-on-nehari} to $tw$ and $u$
yields
\begin{equation*}
 c_\mu^{\cN}\le J_\mu(tw)
 =\frac{t^2}{2}\norm{u}_2^2
  +\frac{\mu(p-1)t^{p+1}}{2(p+1)}\norm{u}_{p+1}^{p+1}
 \le J_\mu(u)=c_\mu^{\cN}.
\end{equation*}
Since $\norm{u}_2>0$, equality forces $t=1$. Consequently,
$w\in\cN_\mu$, $J_\mu(w)=c_\mu^{\cN}$, and
\begin{equation}\label{eq:polya-szego-equality}
 \norm{\nabla w}_2=\norm{\nabla u}_2.
\end{equation}

Propositions~\ref{prop:natural-constraint} and \ref{prop:regularity}
show that $w$ is a positive, smooth weak solution. Its radial profile is
nonincreasing, so Proposition~\ref{prop:strict-decrease} gives
$w(x)=U(\abs{x})$, $U'(r)<0$ for $r>0$, and $U(r)\to0$.
In particular,
\begin{equation}\label{eq:rearrangement-critical-set}
 \left|\left\{x\in\R^N:
  0<w(x)<\operatorname*{ess\,sup}_{\R^N}w,
  \nabla w(x)=0\right\}\right|=0.
\end{equation}
We apply the equality case of the P\'olya--Szeg\H{o} inequality
of Brothers and Ziemer \cite{BrothersZiemer1988}; see also \cite{BurchardFerone2015}. 
We use the case in which both the gradient and Lebesgue exponents equal $2$.

Both \(u\) and \(w\) are nonnegative functions in \(H^1(\mathbb R^N)\).
Since \(u,w\in L^2(\mathbb R^N)\), all their positive superlevel sets have finite measure.
Moreover, \eqref{eq:polya-szego-equality} gives equality in the Pólya--Szegő inequality, while
\eqref{eq:rearrangement-critical-set} verifies the required critical-set condition.
Therefore, the equality case of Brothers and Ziemer yields
\[
u=w(\,\cdot-x_0\,)\quad\text{a.e. in }\mathbb R^N
\]
for some \(x_0\in\mathbb R^N\). 
Proposition~\ref{prop:regularity} gives continuity of both functions,
so this identity holds everywhere and proves
\eqref{eq:ground-state-symmetry}.
\end{proof}

\begin{proof}[Proof of Theorem~\ref{thm:existence}]
Lemma~\ref{lem:positive-level} gives $c_\mu^{\cN}>0$.
Propositions~\ref{prop:attainment} and \ref{prop:natural-constraint}
give a nonnegative radial minimizer $u_\mu$ with $J_\mu'(u_\mu)=0$.
Propositions~\ref{prop:regularity} and \ref{prop:strict-decrease}
give its positivity, smoothness, and strict radial monotonicity.
Every nontrivial critical point lies in $\cN_\mu$, so
\[
 c_\mu^{\cN}
 \le\inf_{v\in\mathscr S_\mu}J_\mu(v)
 \le\inf_{v\in\mathscr S_\mu^{\rad}}J_\mu(v)
 \le J_\mu(u_\mu)=c_\mu^{\cN}.
\]
Both infima are attained by $u_\mu$, proving
\eqref{eq:level-equality-main}.

Now let \(u\) be a positive ground state. Since \(J_\mu'(u)=0\) and \(u\neq0\),
we have \(u\in\mathcal N_\mu\). By \eqref{eq:level-equality-main},
\[
J_\mu(u)=c_\mu^\cN,
\]
so \(u\) is a nonnegative Nehari minimizer. Proposition~\ref{prop:positive-symmetry} therefore yields,
after a translation, radial symmetry and strict radial decrease, while Lemma~\ref{lem:radial-bound}
gives the estimate in \eqref{eq:strict-radial-decay-main}.
\end{proof}

\section{Uniform estimates and convergence to the Gaussian}
\label{sec:limit}

Section~\ref{sec:existence} gives a centered radial ground state for
each $\mu>0$, although uniqueness has not yet been established. 
We prove Theorem~\ref{thm:limit} by deriving estimates uniform over
\(\cG_\mu^{\mathrm{rad}}\) and showing that every sequence of centered radial
ground states with \(\mu\downarrow0\) converges to \(U_0\).
We first establish compactness in the energy space, then local smooth convergence and uniform estimates at infinity. These estimates will control the
coefficients of the difference and linearized equations in
Section~\ref{sec:uniqueness-nondegeneracy}.

\subsection{Energy bounds and compactness}

Projection of the Gaussian onto $\cN_\mu$ gives an upper bound for the action. The Pohozaev identities then control the $H^1$ norm, and the Nehari identity controls the Orlicz norm.

\begin{proposition}[Energy and logarithmic-energy bounds]\label{prop:uniform-energy}

There exist \(\mu_1>0\) and \(C=C(N,p,\mu_1)>0\) such that, for every
\(0<\mu\le\mu_1\),
\[
c_\mu\le C,
\]
and, for every \(u\in \cG_\mu^{\mathrm{rad}}\),
\[
\|u\|_{H^1}+\|u\|_V\le C.
\]
Moreover,
\begin{equation}\label{eq:level-limsup}
 \limsup_{\mu\downarrow0}c_\mu\le c_0.
\end{equation}
\end{proposition}

\begin{proof}
Put $M_0=\|U_0\|_2^2$ and $P_0=\|U_0\|_{p+1}^{p+1}$. Let $t_\mu>0$ be
 the unique scaling factor that projects $U_0$ onto the Nehari constraint.  
Since $K_0(U_0)=0$, formula \eqref{eq:fiber-K} gives
\[
    0
    =
    \frac{K_\mu(t_\mu U_0)}{t_\mu^2}
    =
    -M_0\log t_\mu^2
    -
    \mu t_\mu^{p-1}P_0.
\]
Hence $0<t_\mu\le1$, and therefore
\[
    0
    \le
    -\log t_\mu^2
    =
    \frac{\mu P_0}{M_0}t_\mu^{p-1}
    \le
    \frac{\mu P_0}{M_0}.
\]
It follows that
\[
    t_\mu\to1
    \qquad\text{as }\mu\downarrow0.
\]

Since $t_\mu U_0\in\mathcal N_\mu$ and
$c_\mu=\inf_{\mathcal N_\mu}J_\mu$, we have
\[
    c_\mu
    \le
    J_\mu(t_\mu U_0).
\]
Moreover,
\[
    J_\mu(t_\mu U_0)
    =
    J_0(t_\mu U_0)
    -
    \frac{\mu t_\mu^{p+1}}{p+1}
    \|U_0\|_{p+1}^{p+1}
    \longrightarrow
    J_0(U_0)
    =
    c_0.
\]
Consequently,
\[
    \limsup_{\mu\downarrow0}c_\mu\le c_0.
\]
After decreasing $\mu_1>0$ if necessary, this also yields
\[
    c_\mu\le C,
    \qquad
    0<\mu\le\mu_1.
\]

Now let $u\in \cG_\mu^{\rm rad}$ with $0<\mu\le\mu_1$.
Since $J_\mu(u)=c_\mu$, Lemma~\ref{lem:pohozaev}
gives
\begin{equation}
    \|\nabla u\|_2^2
    =
    Nc_\mu,
    \qquad
    c_\mu
    =
    \frac12\|u\|_2^2
    +
    \frac{\mu(p-1)}{2(p+1)}
    \|u\|_{p+1}^{p+1}.
    \label{eq:ground-identities-repeat}
\end{equation}
Hence
\[
    \|\nabla u\|_2^2
    =
    Nc_\mu
    \le C,
\]
while the second identity in
\eqref{eq:ground-identities-repeat} gives
\[
    \|u\|_2^2
    \le
    2c_\mu
    \le C.
\]
Thus
\[
    \|u\|_{H^1}\le C
\]
uniformly for $0<\mu\le\mu_1$ and
$u\in \cG_\mu^{\rm rad}$.

It remains to control the Orlicz norm.
Since $u\in\mathcal N_\mu$ and
\[
    u^2\log u^2=B(u)-A(|u|),
\]
the Nehari identity yields
\[
\begin{aligned}
    \int_{\mathbb R^N}A(|u|)\,dx
    &=
    \int_{\mathbb R^N}B(u)\,dx
    -
    \|\nabla u\|_2^2
    +
    \mu\|u\|_{p+1}^{p+1} \\
    &\le
    C\|u\|_r^r
    +
    \mu\|u\|_{p+1}^{p+1},
\end{aligned}
\]
for any fixed $r\in(2,2^*)$.
By the uniform $H^1$-bound, the Sobolev embedding, and
$\mu\le\mu_1$, the right-hand side is uniformly bounded.
Therefore
\[
    \rho_A(u)
    =
    \int_{\mathbb R^N}A(|u|)\,dx
    \le C.
\]
The modular--norm estimate in
Proposition~\ref{prop:orlicz-facts} then gives
\[
    \|u\|_V\le C.
\]
Consequently,
\[
    \|u\|_{H^1}+\|u\|_V\le C,
\]
which completes the proof.
\end{proof}

\begin{lemma}[Uniform nonvanishing]\label{lem:uniform-nonvanishing}

After decreasing the constant \(\mu_1\) from Proposition~\ref{prop:uniform-energy} if necessary,
there exists \(d=d(N,p,\mu_1)>0\) such that
\begin{equation*}
 \norm{\nabla u}_2^2\ge d
\end{equation*}
for every \(0<\mu\le\mu_1\) and every
\(u\in \cG_\mu^{\mathrm{rad}}\).

\end{lemma}

\begin{proof}
Put $a=\norm{\nabla u}_2^2$.  From
\eqref{eq:ground-identities-repeat},
$\norm{u}_2^2\le \frac{2a}{N}$; hence
$\norm{u}_{H^1}^2\le(1+\frac 2 N)a$.  Choose
$r\in(2,2^*)$.  Discarding the negative part of the logarithm in the
Nehari identity and using \eqref{eq:positive-log-growth} gives
\[
 a\le C\norm{u}_r^r+\mu_1\norm{u}_{p+1}^{p+1}
 \le C a^{\frac r 2}+C a^{\frac{p+1}{2}}.
\]
Since \(a>0\), \(r>2\) and \(p>1\), division by \(a\) gives
\[
1\le C a^{\frac{r-2}{2}}+C a^{\frac{p-1}{2}}\rightarrow 0\quad (a\rightarrow0).
\] Hence there exists \(d>0\), independent of
\(0<\mu\le\mu_1\) and \(u\in \cG_\mu^{\mathrm{rad}}\), such that
\(a=\|\nabla u\|_2^2\ge d\).
\end{proof}

\begin{proposition}[Compactness and identification of the limit]
\label{prop:every-selection}
If $\mu_n\downarrow0$ and $u_n\in\cG_{\mu_n}^{\rad}$ are chosen
arbitrarily, then we have
\begin{equation}\label{eq:sequence-W-convergence}
 c_{\mu_n}\to c_0,\qquad u_n\to U_0\quad\text{strongly in }W.
\end{equation}
\end{proposition}

\begin{proof}
We first prove convergence along a subsequence and then show that the
same limits hold for the entire sequence. Discarding finitely many terms,
we may assume that $0<\mu_n\le\mu_1$, where $\mu_1$ is chosen as in
Proposition~\ref{prop:uniform-energy} and
Lemma~\ref{lem:uniform-nonvanishing}.

\smallskip
\noindent\emph{Step 1: Compactness and nonvanishing of the weak limit.}
Proposition~\ref{prop:uniform-energy} gives a constant $M>0$ such that
\[
 \|u_n\|_{H^1}+\|u_n\|_V\le M\qquad\text{for every }n.
\]
By Proposition~\ref{prop:orlicz-facts}{\rm(ii)}, $W$ is reflexive.
For each $Q\in O(N)$, composition with $Q$ is a continuous linear
isometry on $W$. Thus $W_{\rad}$, the intersection of the kernels of
the maps $v\mapsto v\circ Q-v$, is a closed linear subspace of $W$
and is also reflexive. We may therefore extract a subsequence, still
denoted by $(u_n)$, and find $u\in W_{\rad}$ such that
\[
 u_n\rightharpoonup u\quad\text{in }W_{\rad}.
\]
The continuous inclusion $W_{\rad}\hookrightarrow H^1(\R^N)$ also gives
weak convergence in $H^1(\R^N)$. Proposition~\ref{prop:radial-compactness}
now yields
\begin{equation}\label{eq:weak-W-limit}
 u_n\to u\quad\text{strongly in }L^2(\R^N)
 \text{ and in }L^r(\R^N)
 \quad\text{for every }2<r<2^*.
\end{equation}
After a further extraction we also have $u_n(x)\to u(x)$ for almost
every $x\in\R^N$. Indeed, strong $L^2$ convergence allows us to choose
indices $n_j$ with $\|u_{n_j}-u\|_2^2\le2^{-j}$; Tonelli's theorem then
gives $\sum_j|u_{n_j}(x)-u(x)|^2<\infty$ almost everywhere.
We relabel this subsequence. Since every $u_n$ is positive,
$u\ge0$ almost everywhere.

The assumption $1<p<2^*-1$ gives $2<p+1<2^*$, so the Sobolev embedding
and the uniform $H^1$ bound imply $\|u_n\|_{p+1}\le C$. Consequently,
\begin{equation}\label{eq:power-vanishing}
 0\le\mu_n\|u_n\|_{p+1}^{p+1}\le C\mu_n\longrightarrow0.
\end{equation}
Let $d>0$ be the constant in Lemma~\ref{lem:uniform-nonvanishing}.
The two identities in \eqref{eq:ground-identities-repeat} give
\[
 c_{\mu_n}=\frac1N\|\nabla u_n\|_2^2\ge\frac dN,
 \qquad
 \|u_n\|_2^2
 =2c_{\mu_n}-\frac{\mu_n(p-1)}{p+1}\|u_n\|_{p+1}^{p+1}.
\]
Taking the lower limit in the second identity, using
\eqref{eq:power-vanishing} and strong $L^2$ convergence, we obtain
\[
 \|u\|_2^2=\lim_{n\to\infty}\|u_n\|_2^2
 \ge\frac{2d}{N}>0.
\]
In particular, the weak limit is nonzero.

\smallskip
\noindent\emph{Step 2: Passage to the limit in the equation.}
Define $h(s)=s\log s^2$ for $s\ne0$ and $h(0)=0$. This function is
continuous on $\R$, since $s\log s^2\to0$ as $s\to0$.
Fix $\delta\in(0,1)$ and $r\in(2,2^*)$. Choose a finite $q>1$ such that
\[
 q(1-\delta)<2,\qquad q(r-1)<2^*.
\]
Such a choice is possible by taking $q$ sufficiently close to $1$;
when $N=2$, only the first inequality restricts this choice.
Set $t=\max\{2,q(r-1)\}$. Then $2\le t<2^*$, with $t<\infty$
when $N=2$, and $q(1-\delta)\le t$.
For any fixed measurable set $E\subset\R^N$ of finite measure,
\eqref{eq:log-vitali-growth} and the inequality
$a^\gamma\le1+a^t$ for $a\ge0$ and $0<\gamma\le t$ imply
\[
 \begin{aligned}
 \int_E|h(u_n)|^q\,dx
 &\le C\int_E\bigl(|u_n|^{q(1-\delta)}
                         +|u_n|^{q(r-1)}\bigr)\,dx\\
 &\le C\bigl(|E|+\|u_n\|_t^t\bigr)\le C_E.
 \end{aligned}
\]
The last bound is uniform in $n$ by the Sobolev embedding and the
uniform $H^1$ bound. For every measurable $F\subset E$, H\"older's
inequality therefore gives
\[
 \sup_n\int_F|h(u_n)|\,dx
 \le C_E^{\frac 1 q}|F|^{1-\frac 1 q}\longrightarrow0
 \qquad\text{as }|F|\to0.
\]
Thus $(h(u_n))$ is uniformly integrable on $E$. Continuity of $h$ and
almost-everywhere convergence of $u_n$ give $h(u_n)\to h(u)$ almost
everywhere. Fatou's lemma also gives $h(u)\in L^q(E)$. Since $|E|<\infty$,
Vitali's convergence theorem applies and yields
\[
 h(u_n)\to h(u)\quad\text{in }L^1(E).
\]
In particular, this convergence holds on every compact subset of $\R^N$.

The power term converges to zero in a global Lebesgue space:
\[
 \|\mu_n u_n^p\|_{L^{\frac{p+1}{p}}(\R^N)}
 =\mu_n\|u_n\|_{p+1}^p\le C\mu_n\longrightarrow0.
\]
For $\varphi\in C_c^\infty(\R^N)$, the equation satisfied by $u_n$ is
\[
 \int_{\R^N}\nabla u_n\cdot\nabla\varphi\,dx
 =\int_{\R^N}h(u_n)\varphi\,dx
   +\mu_n\int_{\R^N}u_n^p\varphi\,dx.
\]
Weak $H^1$ convergence gives convergence of the left-hand side to
$\int_{\R^N}\nabla u\cdot\nabla\varphi\,dx$. If
$K=\operatorname{supp}\varphi$, the logarithmic term satisfies
\[
 \left|\int_{\R^N}\bigl(h(u_n)-h(u)\bigr)\varphi\,dx\right|
 \le\|\varphi\|_\infty\|h(u_n)-h(u)\|_{L^1(K)}\longrightarrow0,
\]
and H\"older's inequality with conjugate exponents $\frac{p+1}{p}$ and $p+1$
gives
\[
 \left|\mu_n\int_{\R^N}u_n^p\varphi\,dx\right|
 \le\mu_n\|u_n\|_{p+1}^p\|\varphi\|_{p+1}\longrightarrow0.
\]
Hence
\[
 \int_{\R^N}\nabla u\cdot\nabla\varphi\,dx
 =\int_{\R^N}u\log u^2\,\varphi\,dx
 \qquad\text{for every }\varphi\in C_c^\infty(\R^N),
\]
which is precisely $-\Delta u=u\log u^2$ in $\mathcal D'(\R^N)$.

\smallskip
\noindent\emph{Step 3: Positivity and the limiting energy identity.}
We have proved that $u\in W$ is a nonzero, nonnegative distributional
solution with parameter $\mu=0$. All the hypotheses of
Proposition~\ref{prop:regularity} are therefore satisfied, so
$u\in C^\infty(\R^N)$ and $u>0$ everywhere. The equation then holds
pointwise, since both sides are continuous.

To justify testing this equation with $u$, recall from
Proposition~\ref{prop:nonlinear-calculus}{\rm(iii)} that
$J_0'(u)$ is a continuous linear functional on $W$, with the formula
\eqref{eq:action-derivative}. The distributional identity says that
$J_0'(u)[\varphi]=0$ for every $\varphi\in C_c^\infty(\R^N)$.
For any $v\in W$, Proposition~\ref{prop:orlicz-facts}{\rm(iv)} supplies
$\varphi_j\in C_c^\infty(\R^N)$ with $\varphi_j\to v$ in $W$.
Continuity gives
$J_0'(u)[v]=\lim_jJ_0'(u)[\varphi_j]=0$.
Taking $v=u$ recovers the identity in
Lemma~\ref{lem:test-by-solution}, namely
\[
 \|\nabla u\|_2^2=\int_{\R^N}u^2\log u^2\,dx.
\]
This integral is absolutely convergent by $u\in W$ and
\eqref{eq:W-log-equivalence}. Substitution into the definition
\eqref{eq:action-intro} gives
\[
 J_0(u)
 =\frac12\|\nabla u\|_2^2+\frac12\|u\|_2^2
   -\frac12\int_{\R^N}u^2\log u^2\,dx
 =\frac12\|u\|_2^2.
\]
On the other hand, the second identity in
\eqref{eq:ground-identities-repeat} gives
\[
 c_{\mu_n}
 =\frac12\|u_n\|_2^2
  +\frac{\mu_n(p-1)}{2(p+1)}\|u_n\|_{p+1}^{p+1}
 \longrightarrow\frac12\|u\|_2^2,
\]
where we used \eqref{eq:weak-W-limit} and \eqref{eq:power-vanishing}.
Thus the energy levels converge along the extracted subsequence, and
\begin{equation}\label{eq:limit-energy-comparison}
 J_0(u)=\frac12\|u\|_2^2=\lim_{n\to\infty}c_{\mu_n}.
\end{equation}

\smallskip
\noindent\emph{Step 4: Identification of the limit.}
The smooth representative of $u$ is radial about the origin: for each
$Q\in O(N)$, $u\circ Q=u$ almost everywhere by $u\in W_{\rad}$, and
continuity makes this equality valid everywhere. Applying
Lemma~\ref{lem:radial-bound} to this representative gives
\[
 |u(x)|^2\le C_N|x|^{1-N}\|u\|_2\|\nabla u\|_2
 \qquad(x\ne0).
\]
Since $N\ge2$, this proves $u(x)\to0$ as $|x|\to\infty$.
Together with continuity on bounded sets, this also proves that $u$
is bounded on $\R^N$.

For $N\ge3$, we can now apply
\cite[Theorem~1.2]{dAveniaMontefuscoSquassina2014}: $u\in H^1(\R^N)$
is a strictly positive $C^2$ solution of the logarithmic equation
and tends to zero at infinity. We use that theorem with $\omega=0$;
its validity for every $\omega\in\R$ is stated in the introduction
of the cited paper. It yields a vector $a\in\R^N$ such that
\[
 u(x)=\exp\!\left(\frac{N-|x-a|^2}{2}\right).
\]

For $N=2$, we reduce the classification to dimension three. Set
\[
 g(z)=\exp\!\left(\frac{1-z^2}{2}\right),
 \qquad \widetilde u(x,z)=u(x)g(z),
 \qquad (x,z)\in\R^2\times\R.
\]
Differentiation gives
\[
 g'(z)=-zg(z),\qquad
 g''(z)=(z^2-1)g(z),\qquad
 -g''(z)=g(z)\log g(z)^2.
\]
Using the equations for $u$ and $g$, and their strict positivity,
we compute
\[
 \begin{aligned}
 -\Delta_{x,z}\widetilde u
 &=g(-\Delta_xu)-ug''\\
 &=ug\log u^2+ug\log g^2\\
 &=ug\log(u^2g^2)
  =\widetilde u\log\widetilde u^2.
 \end{aligned}
\]
The product is positive and smooth. Since $g,g'\in L^2(\R)$,
Tonelli's theorem gives
\[
 \begin{aligned}
 \|\widetilde u\|_{L^2(\R^3)}^2
 &=\|u\|_{L^2(\R^2)}^2\|g\|_{L^2(\R)}^2,\\
 \|\nabla_{x,z}\widetilde u\|_{L^2(\R^3)}^2
 &=\|\nabla_xu\|_{L^2(\R^2)}^2\|g\|_{L^2(\R)}^2
   +\|u\|_{L^2(\R^2)}^2\|g'\|_{L^2(\R)}^2.
 \end{aligned}
\]
Both quantities are finite, so $\widetilde u\in H^1(\R^3)$.
To verify decay in all directions, note that
$|(x,z)|>\sqrt2R$ implies either $|x|>R$ or $|z|>R$. Hence
\[
 \sup_{|(x,z)|>\sqrt2R}\widetilde u(x,z)
 \le\max\left\{
   \|g\|_\infty\sup_{|x|>R}u(x),\,
   \|u\|_\infty\sup_{|z|>R}g(z)
 \right\}\longrightarrow0
 \quad(R\to\infty).
\]
The same classification theorem, now in dimension three with
$\omega=0$, gives $a\in\R^2$ and $b\in\R$ such that
\[
 \widetilde u(x,z)
 =\exp\!\left(\frac{3-|x-a|^2-(z-b)^2}{2}\right).
\]
By the product definition,
$\partial_z\log\widetilde u=\frac{g'}{g}=-z$, whereas the displayed Gaussian
has logarithmic derivative $-(z-b)$. Equality for every $z$ gives
$b=0$. Division by $g(z)>0$ therefore yields
\[
 u(x)=\exp\!\left(\frac{2-|x-a|^2}{2}\right).
\]

Thus, in every dimension under consideration, $u$ is a translate of
$U_0$. Radiality about the origin gives $u(x)=u(-x)$ for all $x$.
In the Gaussian formula this implies
$|x-a|^2=|x+a|^2$, hence $4x\cdot a=0$ for every $x$.
Taking $x=a$ proves $a=0$, and therefore $u=U_0$.
Equations~\eqref{eq:limit-energy-comparison} and
\eqref{eq:gaussian-energy} now give $c_{\mu_n}\to c_0$ along the
extracted subsequence.

\smallskip
\noindent\emph{Step 5: Strong convergence in $H^1$.}
The first identity in \eqref{eq:ground-identities-repeat} gives
\[
 \|\nabla u_n\|_2^2=Nc_{\mu_n}
 \longrightarrow Nc_0=\|\nabla U_0\|_2^2.
\]
Here the last equality follows from
\eqref{eq:pohozaev-ground-identities} applied to the positive solution
$U_0\in W$ with $\mu=0$, together with $J_0(U_0)=c_0$.
Weak $H^1$ convergence to $U_0$ implies
\[
 \int_{\R^N}\nabla u_n\cdot\nabla U_0\,dx
 \longrightarrow\|\nabla U_0\|_2^2.
\]
Expanding the square therefore gives
\[
 \begin{aligned}
 \|\nabla u_n-\nabla U_0\|_2^2
 &=\|\nabla u_n\|_2^2+\|\nabla U_0\|_2^2
   -2\int_{\R^N}\nabla u_n\cdot\nabla U_0\,dx
 \longrightarrow0.
 \end{aligned}
\]
Combining this with strong $L^2$ convergence in
\eqref{eq:weak-W-limit} proves $\|u_n-U_0\|_{H^1}\to0$.

\smallskip
\noindent\emph{Step 6: Strong convergence in $V$ and removal of the subsequence.}
Fix $r\in(2,2^*)$. By \eqref{eq:weak-W-limit},
$\|u_n-U_0\|_r\to0$, while $(\|u_n\|_r)$ is bounded.
Using \eqref{eq:B-mean-value} and H\"older's inequality with exponents
$\frac r{r-1}$ and $r$, we obtain
\[
 \begin{aligned}
 \|B(u_n)-B(U_0)\|_1
 &\le C_r\int_{\R^N}
       \bigl(|u_n|^{r-1}+|U_0|^{r-1}\bigr)|u_n-U_0|\,dx\\
 &\le C_r\bigl(\|u_n\|_r^{r-1}+\|U_0\|_r^{r-1}\bigr)
               \|u_n-U_0\|_r\longrightarrow0.
 \end{aligned}
\]
In particular, the integrals of $B(u_n)$ converge to that of $B(U_0)$.
The Nehari identity \eqref{eq:nehari-solution} for $u_n$, together
with Step~5 and \eqref{eq:power-vanishing}, gives
\[
 \begin{aligned}
 \int_{\R^N}u_n^2\log u_n^2\,dx
 &=\|\nabla u_n\|_2^2-\mu_n\|u_n\|_{p+1}^{p+1}\\
 &\longrightarrow\|\nabla U_0\|_2^2
  =\int_{\R^N}U_0^2\log U_0^2\,dx,
 \end{aligned}
\]
where the last equality is \eqref{eq:limit-nehari-identity} for $U_0$.
All these logarithmic integrals are absolutely convergent by
\eqref{eq:W-log-equivalence}. Since
$A(|s|)=B(s)-s^2\log s^2$, subtraction of the two convergent integrals
yields
\[
 \begin{aligned}
 \rho_A(u_n)
 &=\int_{\R^N}B(u_n)\,dx
   -\int_{\R^N}u_n^2\log u_n^2\,dx\\
 &\longrightarrow
   \int_{\R^N}B(U_0)\,dx
   -\int_{\R^N}U_0^2\log U_0^2\,dx
  =\rho_A(U_0)<\infty.
 \end{aligned}
\]
We have already established $u_n\to U_0$ almost everywhere, and
$U_0\in V$ because $U_0=u\in W$. Thus both hypotheses of the modular
convergence criterion in Proposition~\ref{prop:orlicz-facts}{\rm(iii)}
hold, and it gives $\|u_n-U_0\|_V\to0$.
By the definition of the norm in Definition~\ref{def:energy-space},
\[
 \|u_n-U_0\|_W
 =\|u_n-U_0\|_{H^1}+\|u_n-U_0\|_V\longrightarrow0.
\]

Finally, the argument above applies to any subsequence of the original
sequence: its parameters still tend to zero, and all the bounds used
above are independent of the chosen ground states. Every subsequence
therefore has a further subsequence along which both
$c_{\mu_n}\to c_0$ and $u_n\to U_0$ strongly in $W$.
If \eqref{eq:sequence-W-convergence} failed for the original sequence,
there would be $\varepsilon>0$ and a subsequence such that
\[
 |c_{\mu_n}-c_0|+\|u_n-U_0\|_W\ge\varepsilon
\]
at every index of that subsequence. A further subsequence with the two
convergences just proved would make the left-hand side tend to zero,
which is a contradiction. This proves
\eqref{eq:sequence-W-convergence} for the entire sequence.
\end{proof}

\begin{corollary}[Uniform convergence of radial ground states]
\label{cor:set-uniform-W}
The two limits in \eqref{eq:set-uniform-W-main} hold.
\end{corollary}

\begin{proof}
  If either assertion in \eqref{eq:set-uniform-W-main} failed, one could choose a sequence
\(\mu_n\downarrow0\) and, in the second case, ground states
\(u_n\in \cG_{\mu_n}^{\mathrm{rad}}\) for which the corresponding convergence fails.
This contradicts Proposition~\ref{prop:every-selection}.
\end{proof}

\subsection{Local regularity and global convergence}

Uniform boundedness follows from the iteration estimate in Lemma~\ref{lem:moser-iteration}. Positivity of the Gaussian on each fixed ball then permits smooth elliptic bootstrapping.

\begin{proposition}[Uniform boundedness and local smooth convergence]
\label{prop:uniform-local}

After decreasing \(\mu_1\) if necessary, the \(L^\infty\)-bound in
\eqref{eq:uniform-W-Linfty-main} holds. Moreover, for every \(R>0\), every integer \(k\ge0\)
and every \(\alpha\in(0,1)\),
\[
\lim_{\mu\downarrow0}
\sup_{u\in \cG_\mu^{\mathrm{rad}}}
\|u-U_0\|_{C^{k,\alpha}(B_R)}=0.
\]
In addition, \eqref{eq:local-smooth-main} holds for every \(R>0\), \(k\in\mathbb N_0\),
and \(\alpha\in(0,1)\).
\end{proposition}

\begin{proof}
The uniform bound in $W$ follows from
Proposition~\ref{prop:uniform-energy}.  The global Moser iteration in
Lemma~\ref{lem:moser-iteration} applies because the $H^1$ norms of the solutions and the parameter $\mu$ are uniformly bounded; it gives the asserted
$L^\infty$ bound.  The truncations and exponent choices for the iteration
scheme of \cite{Moser1961} are given in the appendix.

The right-hand side
$u\log u^2+\mu u^p$ is now uniformly bounded.  Interior estimates on
every pair $B_1(y)\Subset B_2(y)$ give, for every $1<q<\infty$,
\begin{equation}\label{eq:uniform-local-W2q}
 \sup_{0<\mu\le\mu_1}\sup_{u\in\cG_\mu^{\rad}}\sup_{y\in\R^N}
 \norm{u}_{W^{2,q}(B_1(y))}\le C_q.
\end{equation}

Fix \(R>0\) and \(\beta\in(0,1)\). Choose \(q>N\) sufficiently large so that
\[
W^{2,q}(B_{R+1})\hookrightarrow\hookrightarrow C^{1,\beta}(B_R).
\]
For any sequence \(\mu_n\downarrow0\) and
\(u_n\in \cG_{\mu_n}^{\mathrm{rad}}\), Proposition~\ref{prop:every-selection} and \eqref{eq:uniform-local-W2q}
therefore imply
\[
u_n\to U_0\qquad\text{in }C^{1,\beta}(B_R).
\]
 Uniqueness of the limit gives convergence of the whole sequence.
Since $\min_{\overline{B_{R+1}}}U_0>0$, the solutions have a common
positive lower bound on $B_{R+1}$ for all sufficiently large $n$.  On that fixed positive range,
$s\mapsto s\log s^2+\mu_ns^p$ and all derivatives needed for a fixed
$k$ are uniformly controlled. 

Interior Schauder estimates on nested balls give uniform bounds of every fixed order. 
For given $k$ and $\alpha$, choose $\alpha<\beta<1$ and use a uniform $C^{k+1,\beta}$ bound
 on a slightly larger ball. Compact 
embedding into $C^{k,\alpha}(\overline{B_R})$ and the already identified local limit imply
\[
u_n\to U_0\qquad\text{in }C^{k,\alpha}(B_R).
\]
As this holds for every
sequence $\mu_n\downarrow0$ and every choice
$u_n\in\cG_{\mu_n}^{\rad}$, it gives \eqref{eq:local-smooth-main}.
\end{proof}

The radial estimate makes the solutions uniformly small at infinity. The H\"older control of the logarithmic term in
Lemma~\ref{lem:holder-tail} converts this smallness into the right-hand-side estimates needed for Schauder theory.

\begin{proposition}[Uniform far-field regularity]\label{prop:uniform-tail}
The algebraic estimate \eqref{eq:uniform-algebraic-main} holds. Moreover, for every $\alpha\in(0,1)$, the uniform limit \eqref{eq:uniform-tail-main} also holds.
\end{proposition}

\begin{proof}
The uniform $H^1$ bound and Lemma~\ref{lem:radial-bound} give
\eqref{eq:uniform-algebraic-main}. Hence
\begin{equation}\label{eq:uniform-C0-tail}
 \lim_{R\to\infty}\sup_{0<\mu\le\mu_1}
 \sup_{u\in\cG_\mu^{\rad}}\sup_{\abs{x}\ge R}u(x)=0.
\end{equation}
Estimate \eqref{eq:uniform-local-W2q}, with $q>N$, gives a uniform
Lipschitz bound on balls of every fixed radius, with constants independent of their centers.
 In particular, it applies on $B_2(y)$.  For every $\sigma\in(0,1)$, the
map $h(s)=s\log s^2$, $h(0)=0$, is $\sigma$-H\"older on each fixed
$[0,M]$, while $s\mapsto s^p$ is Lipschitz there.  Consequently the right-hand sides 
are uniformly $C^{0,\sigma}$ on $B_2(y)$..

Fix $\alpha<\sigma<1$.  
By \eqref{eq:uniform-C0-tail} and the continuity of
\(s\mapsto s\log s^2\) at \(s=0\), the \(L^\infty\)-norms of the right-hand sides
\[
u\log u^2+\mu u^p
\]
on balls centered at \(y\) tend uniformly to zero as \(|y|\to\infty\).
Since these right-hand sides have a uniform \(C^{0,\sigma}\)-bound, the interpolation
inequality \eqref{eq:holder-interpolation} implies that their \(C^{0,\alpha}\)-norms also tend uniformly to zero
for every \(0<\alpha<\sigma<1\).

The interior Schauder estimate on
$B_1(y)\Subset B_2(y)$ gives
\[
 \norm{u}_{C^{2,\alpha}(B_1(y))}
 \le C\bigl(\norm{u}_{C^0(B_2(y))}
 +\norm{u\log u^2+\mu u^p}_{C^{0,\alpha}(B_2(y))}\bigr),
\]
where $C=C(N,\alpha)$ is independent of $u,\mu,y$.  
Both terms therefore tend to zero as \(|y|\to\infty\), uniformly with respect to
\(0<\mu\le\mu_1\) and \(u\in \cG_\mu^{\mathrm{rad}}\), proving
\eqref{eq:uniform-tail-main}.
\end{proof}

\begin{proposition}[Global H\"older convergence]\label{prop:global-holder}
For every $\alpha\in(0,1)$, \eqref{eq:global-smooth-main} holds.
\end{proposition}

\begin{proof}
Take an arbitrary sequence $\mu_n\downarrow0$ and
$u_n\in\cG_{\mu_n}^{\rad}$.  On each fixed ball,
Proposition~\ref{prop:uniform-local} gives convergence in a
$C^{2,\beta}$ norm with a fixed $\alpha<\beta<1$.  

Choose \(R>0\) sufficiently large. By Proposition~\ref{prop:uniform-tail}, with \(\beta\) in place of
\(\alpha\), the local \(C^{2,\beta}\)-norms of \(u_n\) are uniformly small outside
\(B_R\). The same is true for \(U_0\), since the Gaussian and all its derivatives
vanish at infinity.
Combining these exterior estimates with the local convergence from Proposition~\ref{prop:uniform-local} gives
\[
\|D^j(u_n-U_0)\|_\infty\to0,
\qquad j=0,1,2.
\]
To control the remaining seminorm, set $f_n=D^2(u_n-U_0)$ and use

\[\frac{\abs{f_n(x)-f_n(y)}}{\abs{x-y}^\alpha}
 \le [f_n]_{C^{0,\beta}(B_1(x))}\abs{x-y}^{\beta-\alpha},\quad
 0<\abs{x-y}<1,\]
 and
 \[\frac{\abs{f_n(x)-f_n(y)}}{\abs{x-y}^\alpha}
 \le 2\norm{f_n}_\infty,\quad
 \abs{x-y}\ge1.\]

Consider first \(0<|x-y|<1\). If one of the two points lies outside
\(B_{R+2}\), relabel the pair so that \(x\notin B_{R+2}\). Then
\(B_1(x)\) lies in the far field, and the first estimate above is uniformly small
by Proposition~\ref{prop:uniform-tail}. 

If both points lie in \(B_{R+2}\), then
\(x,y\in B_{R+3}\), and the same quotient is controlled by the local
\(C^{2,\beta}\)-convergence from Proposition~\ref{prop:uniform-local}. For \(|x-y|\ge1\),
the second estimate and \(\|f_n\|_\infty\to0\) apply.
Therefore
\[
[f_n]_{C^{0,\alpha}(\mathbb R^N)}\to0,
\]
which proves the convergence in \eqref{eq:global-holder-norm} and
 hence \eqref{eq:global-smooth-main}.
\end{proof}

\begin{proof}[Proof of Theorem~\ref{thm:limit}]
Proposition~\ref{prop:every-selection} and
Corollary~\ref{cor:set-uniform-W} give the energy-space limit and
$c_\mu\to c_0$ in \textup{(i)}.
The bounds in \eqref{eq:uniform-W-Linfty-main} follow from
Propositions~\ref{prop:uniform-energy} and \ref{prop:uniform-local}, respectively,
while Proposition~\ref{prop:uniform-tail} gives
\eqref{eq:uniform-algebraic-main} and \eqref{eq:uniform-tail-main}.  Finally,
Propositions~\ref{prop:uniform-local} and \ref{prop:global-holder}
give the local and global convergence in
\eqref{eq:local-smooth-main} and \eqref{eq:global-smooth-main}. This completes the proof.
\end{proof}

\section{Uniqueness and nondegeneracy of ground states}
\label{sec:uniqueness-nondegeneracy}
We prove Theorem~\ref{thm:nondegeneracy} using the Gaussian linearization.
The coefficient estimates from Section~\ref{sec:limit}, combined with a cutoff energy argument,
prevent normalized difference and kernel sequences from vanishing locally.
Compactness then yields nonzero limits, whose radiality or orthogonality to the
translation modes contradicts the Gaussian kernel characterization.
Section~\ref{subsec:uniqueness-translations} combines radial uniqueness
with the symmetry proved in Section~\ref{sec:existence} to obtain
uniqueness up to translations.

\subsection{Uniqueness by the normalized-difference method}
\label{sec:uniqueness}

Subtracting the equations for two ground states gives a linear equation
for their difference. Divided differences extend this identity to points where the two solutions agree and allow the estimates of
Section~\ref{sec:limit} to be applied to its coefficient.

For $s,t>0$ define the two divided differences, including their diagonal values, by
\begin{equation}\label{eq:divided-differences}
 D_{\log}(s,t)=
 \begin{cases}
 \dfrac{s\log s^2-t\log t^2}{s-t},&s\ne t,\\[2mm]
 \log s^2+2,&s=t,
 \end{cases}
 \qquad
 D_p(s,t)=
 \begin{cases}
 \dfrac{s^p-t^p}{s-t},&s\ne t,\\[2mm]
 ps^{p-1},&s=t.
 \end{cases}
\end{equation}
The fundamental theorem of calculus gives the formulas
\begin{equation}\label{eq:divided-difference-integrals}
 D_{\log}(s,t)=\int_0^1\bigl(\log(\theta s+(1-\theta)t)^2+2\bigr)\,d\theta,
 \quad
 D_p(s,t)=p\int_0^1(\theta s+(1-\theta)t)^{p-1}\,d\theta.
\end{equation}

\begin{lemma}[Difference coefficients]\label{lem:difference-coefficients}

Suppose $\mu_n\downarrow 0$ and 
\[
u_n,v_n\in \cG_{\mu_n}^{\mathrm{rad}}.
\]
Define
\begin{equation}\label{eq:difference-coefficient}
 a_n(x)=D_{\log}(u_n(x),v_n(x))
       +\mu_nD_p(u_n(x),v_n(x)).
\end{equation}
Then the following properties hold:

\begin{enumerate}[\rm{(}i\rm{)}]
    \item $a_n$ converges locally uniformly to
    \[
    a_0(x):=N+2-|x|^2;
    \]
    equivalently, for every compact set $K\subset\mathbb{R}^N$,
    \begin{equation}\label{eq:coefficient-compact-limit}
    \|a_n-a_0\|_{L^\infty(K)}\longrightarrow 0.
    \end{equation}

    \item There exist constants $C_0>0$ and sufficiently large $n_0\in\mathbb{N}$ such that
    \[
    a_n(x)\le C_0
    \qquad
    \text{for all }x\in\mathbb{R}^N
    \text{ and }n\ge n_0.
    \]

    \item For every $A>0$, there exist $R_A>0$ and
    $n_A\in\mathbb{N}$ such that
    \begin{equation}\label{eq:coefficient-negative-tail}
    a_n(x)\le -A
    \qquad
    \text{whenever }|x|\ge R_A
    \text{ and }n\ge n_A.
    \end{equation}
\end{enumerate}

\end{lemma}

\begin{proof}
For $\theta\in[0,1]$, set
\[
    z_{n,\theta}(x)
    :=
    \theta u_n(x)+(1-\theta)v_n(x).
\]
By the integral representations in \eqref{eq:divided-difference-integrals},
\[
    D_{\log}(u_n,v_n)
    =
    \int_0^1
    \bigl(\log z_{n,\theta}^2+2\bigr)\,d\theta,
\]
and
\[
    D_p(u_n,v_n)
    =
    p\int_0^1
    z_{n,\theta}^{p-1}\,d\theta.
\]

Let $K\subset\mathbb R^N$ be compact.
By Proposition~\ref{prop:uniform-local},
\[
    u_n,v_n\to U_0
    \qquad\text{uniformly on }K.
\]
Since $U_0>0$ and
\[
    m_K:=\min_K U_0>0,
\]
for all sufficiently large $n$,
\[
    z_{n,\theta}(x)\ge \frac{m_K}{2}
\]
uniformly for $x\in K$ and $\theta\in[0,1]$.
Moreover,
\[
    z_{n,\theta}\to U_0
\]
uniformly on $K$, uniformly with respect to
$\theta\in[0,1]$.
Hence
\[
    D_{\log}(u_n,v_n)
    \longrightarrow
    \log U_0^2+2
\]
uniformly on $K$.
The same local boundedness gives
\[
    |D_p(u_n,v_n)|\le C_K,
\]
and therefore, since $\mu_n\to0$,
\[
    \mu_nD_p(u_n,v_n)\to0
\]
uniformly on $K$.
Since
\[
    \log U_0^2+2
    =
    N+2-|x|^2,
\]
this proves (i).

We next prove the global upper bound.
By the uniform $L^\infty$ estimate in Theorem~\ref{thm:limit},
there exists $M>0$ such that, for all sufficiently large $n$,
\[
    0<u_n(x),v_n(x)\le M
    \qquad
    \text{for all }x\in\mathbb R^N.
\]
Hence
\[
    0<z_{n,\theta}(x)\le M.
\]
Using the integral formulas above,
\[
    D_{\log}(u_n,v_n)
    \le
    \log M^2+2,
\]
and
\[
    D_p(u_n,v_n)
    \le
    pM^{p-1}.
\]
Since $\mu_n\le\mu_1$ for all sufficiently large $n$,
\[
    a_n(x)
    \le
    \log M^2+2+\mu_1pM^{p-1}.
\]
After enlarging the constant if necessary, this proves (ii).

Finally, fix $A>0$.
Since
\[
    \log\delta^2\to-\infty,
    \qquad
    \delta^{p-1}\to0
    \quad\text{as }\delta\downarrow0,
\]
we may choose $\delta>0$ such that
\[
    \log\delta^2+2
    +
    \mu_1p\delta^{p-1}
    \le -A.
\]
By the uniform far-field smallness in
Theorem~\ref{thm:limit}, there exists $R_A>0$
such that, for all sufficiently large $n$,
\[
    u_n(x),v_n(x)\le\delta
    \qquad
    \text{whenever }|x|\ge R_A.
\]
Therefore
\[
    z_{n,\theta}(x)\le\delta
\]
for every $\theta\in[0,1]$, and the integral representations give
\[
\begin{aligned}
    a_n(x)
    &\le
    \log\delta^2+2
    +
    \mu_np\delta^{p-1} \\
    &\le
    \log\delta^2+2
    +
    \mu_1p\delta^{p-1}
    \le -A.
\end{aligned}
\]
This proves (iii).
\end{proof}

A normalized sequence may converge to zero on every bounded set.
The lemma below justifies the global energy identity for coefficients bounded above and uses their negativity outside a fixed ball to exclude local vanishing.

\begin{lemma}[Nonvanishing of normalized sequences]\label{lem:no-loss}

Let \(C_0>0\), \(A>0\), and \(R>0\). Suppose that
\[
w_n\in L^2(\mathbb R^N)\cap H^1_{\mathrm{loc}}(\mathbb R^N),
\qquad
\|w_n\|_2=1,
\]
and that \(a_n\in L^\infty_{\mathrm{loc}}(\mathbb R^N)\) satisfies
\[
-\Delta w_n=a_n w_n
\qquad\text{in }\mathcal D'(\mathbb R^N).
\]
Assume further that, for every \(n\),
\[
a_n\le C_0 \quad\text{a.e. in }\mathbb R^N,
\qquad
a_n\le -A \quad\text{a.e. in }\mathbb R^N\setminus B_R.
\]
Then \((w_n)\) is bounded in \(H^1(\mathbb R^N)\),
\begin{equation}\label{eq:linear-energy-identity}
 \norm{\nabla w_n}_2^2=\int a_nw_n^2
\end{equation}
and we have
\begin{equation}\label{eq:local-mass-lower}
 \int_{B_R}w_n^2\ge\frac{A}{A+C_0}.
\end{equation}

\end{lemma}

\begin{proof}
The potential may be unbounded below, so
\eqref{eq:linear-energy-identity} is used only after the required integrability has been established.  Put $b_n=C_0-a_n\ge0$, and test
$-\Delta w_n+b_nw_n=C_0w_n$ by $\eta_S^2w_n$, where $\eta_S$ is the
cutoff defined in Section~\ref{sec:preliminaries}.  Young's inequality on
the cross term gives
\begin{equation}\label{eq:weighted-linear-bound}
 \frac12\int_{\R^N}\eta_S^2\abs{\nabla w_n}^2
 +\int_{\R^N} b_n\eta_S^2w_n^2
 \le C_0\norm{w_n}_2^2
 +2\int_{\R^N}\abs{\nabla\eta_S}^2w_n^2.
\end{equation}

Taking \(S\to\infty\) in \eqref{eq:weighted-linear-bound} and using Fatou's lemma gives, uniformly in \(n\),
\[
\|\nabla w_n\|_2\le C,
\qquad
\int_{\mathbb R^N} b_n w_n^2<\infty.
\]
In particular,
\[
a_n w_n^2=C_0w_n^2-b_nw_n^2\in L^1(\mathbb R^N).
\]
Thus all diagonal terms in the cutoff identity are integrable. We may therefore let
\(S\to\infty\): dominated convergence applies to the diagonal terms, while
\[
\left|\int_{\R^N} 2\eta_S w_n\nabla\eta_S\cdot\nabla w_n\right|
\le \frac{C}{S}\|w_n\|_2\|\nabla w_n\|_2\to0.
\]
Consequently, \eqref{eq:linear-energy-identity} follows.
 
 Splitting that
identity over $B_R$ and its complement yields
\[
 0\le\norm{\nabla w_n}_2^2
 \le C_0\int_{B_R}w_n^2-A\int_{\R^N\setminus B_R}w_n^2.
\]
Since the total $L^2$ mass is one, this implies
\eqref{eq:local-mass-lower}.
\end{proof}

Lemma~\ref{lem:no-loss} will give a nonzero $H^1$ distributional
solution of the limiting equation. We characterize the Gaussian kernel in this class so that its spectral description applies directly to the limit.

\begin{lemma}[Gaussian linearized kernel]\label{lem:gaussian-kernel}
  \textnormal{\cite[Theorem~1.3]{dAveniaMontefuscoSquassina2014}}
  For every $N\ge2$.
  \begin{enumerate}[\rm{(}i\rm{)}]
    \item Let
\[
    L_0
    =
    -\Delta+|x|^2-N-2
\]
be realized through the quadratic form of the harmonic oscillator.
Then
\[
    \ker L_0
    =
    \operatorname{span}
    \{\partial_{x_1}U_0,\ldots,\partial_{x_N}U_0\},
\]
and
\[
    \ker\bigl(L_0|_{L^2_{\rm rad}}\bigr)=\{0\}.
\]
\item Moreover, if $w\in H^1(\mathbb R^N)$ satisfies
\[
    L_0w=0
    \qquad\text{in }\mathcal D'(\mathbb R^N),
\]
then
\[
    w
    \in
    \operatorname{span}
    \{\partial_{x_1}U_0,\ldots,\partial_{x_N}U_0\}.
\]
\item In particular, every radial $H^1$ distributional solution is
identically zero.
  \end{enumerate}
\end{lemma}

\begin{proof}
The nondegeneracy of the Gaussian ground state is classical;
see, for instance, \cite[Theorem~1.3]{dAveniaMontefuscoSquassina2014}.
More precisely, since
\[
    U_0(x)=e^{\frac N 2}e^{\frac{-|x|^2} 2},
\]
the result in \cite[Theorem~1.3]{dAveniaMontefuscoSquassina2014}
gives
\[
    \ker L_0
    =
    \operatorname{span}
    \{\partial_{x_1}U_0,\ldots,\partial_{x_N}U_0\}.
\]
Since every nonzero element of this space has the form
\[
    -(a\cdot x)U_0(x),
\]
no nonzero element of $\ker L_0$ is radial. Hence
\[
    \ker\bigl(L_0|_{L^2_{\rm rad}}\bigr)=\{0\}.
\]

It remains only to justify that an $H^1$ distributional
solution belongs to the operator realization above.
Let $w\in H^1(\mathbb R^N)$ satisfy $L_0w=0$ in
$\mathcal D'(\mathbb R^N)$. Testing
\[
    -\Delta w+(|x|^2+C)w=(N+2+C)w
\]
with $\eta_R^2w$, where $\eta_R$ is the standard cutoff,
and using Young's inequality gives
\[
    \frac12\int_{\mathbb R^N}\eta_R^2|\nabla w|^2
    +
    \int_{\mathbb R^N}|x|^2\eta_R^2w^2
    \le C\|w\|_2^2
\]
uniformly in $R$. Letting $R\to\infty$ yields
\[
    |x|w\in L^2(\mathbb R^N).
\]
Thus $w$ belongs to the quadratic-form domain of
$-\Delta+|x|^2$. By density of $C_c^\infty(\mathbb R^N)$
in this form domain, the distributional identity extends to
the whole form domain and hence places $w$ in the operator
domain of $L_0$.Hence $w\in D(L_0)$, and the kernel characterization above applies. Therefore
\[
w\in \operatorname{span}\{\partial_{x_1}U_0,\ldots,\partial_{x_N}U_0\},
\]
which completes the proof.
\end{proof}

\begin{proposition}[Small-parameter radial uniqueness]
\label{prop:radial-uniqueness}
There exists $\mu_{\mathrm{uni}}>0$ such that
$\cG_\mu^{\rad}$ contains exactly one element for
$0<\mu<\mu_{\mathrm{uni}}$.
\end{proposition}

\begin{proof}
Otherwise, there exist $\mu_n\downarrow0$ and distinct
$u_n,v_n\in\cG_{\mu_n}^{\rad}$.  Define the $L^2$-normalized
difference
\begin{equation*}
 d_n=\norm{u_n-v_n}_2>0,\qquad
 w_n=\frac{u_n-v_n}{d_n},\qquad \norm{w_n}_2=1.
\end{equation*}
Subtraction of the two equations, with
\eqref{eq:divided-differences} also defining the value at points where
$u_n=v_n$, gives the exact equation
\begin{equation}\label{eq:difference-equation}
 -\Delta w_n=a_nw_n,
\end{equation}
where $a_n$ is defined by \eqref{eq:difference-coefficient}.

Apply Lemma~\ref{lem:difference-coefficients} with a fixed $A>0$, and
 Lemma~\ref{lem:no-loss}.  The sequence is bounded in $H^1$, and, after passing to a subsequence,
\[
w_n\rightharpoonup w \quad\text{in }H^1(\mathbb R^N),
\qquad
w_n\to w \quad\text{strongly in }L^2_{\mathrm{loc}}(\mathbb R^N).
\]
For the fixed ball appearing in Lemma~\ref{lem:no-loss}, \eqref{eq:local-mass-lower} then gives
\[
\int_{B_R}w^2
 =\lim_{n\to\infty}\int_{B_R}w_n^2
 \ge \frac{A}{A+C_0}>0,
\]
so \(w\neq0\).
Estimate \eqref{eq:local-mass-lower} gives
$w\ne0$.
Since both \(u_n\) and \(v_n\) are radial about the origin, each \(w_n\) is radial.
The radial subspace of \(H^1(\mathbb R^N)\) is weakly closed; hence \(w\) is radial.
The local uniform convergence in
\eqref{eq:coefficient-compact-limit} permits passage to the limit in
\eqref{eq:difference-equation}, and
\[
 L_0w=-\Delta w-(N+2-\abs{x}^2)w=0.
\]
This contradicts the radial-kernel statement in
Lemma~\ref{lem:gaussian-kernel}.
\end{proof}

\subsection{Radial and full nondegeneracy}\label{sec:nondegeneracy}

The same compactness argument applies to kernel functions.
Translation invariance gives kernel functions, and we exclude additional modes by imposing orthogonality to them.

For $0<\mu<\mu_{\mathrm{uni}}$, let $u_\mu$ denote the unique member
of $\cG_\mu^{\rad}$.  Define
\begin{equation*}
 a_\mu(x)=\log u_\mu(x)^2+2+\mu p\,u_\mu(x)^{p-1},
 \qquad V_\mu(x)=-a_\mu(x).
\end{equation*}
The real scalar linearized operator, corresponding to the usual \(L_{+,\mu}\), is
\[
L_\mu\phi
 =-\Delta\phi-(\log u_\mu^2+2)\phi-\mu p u_\mu^{p-1}\phi
 =-\Delta\phi+V_\mu\phi.
\]
Since $u_\mu$ is positive and smooth, $V_\mu$ is locally bounded.
Since $u_\mu(x)\to0$, $V_\mu(x)\to+\infty$.  In particular, it is
bounded from below, so
$(V_\mu)_-=\max\{-V_\mu,0\}\in L^\infty(\R^N)$.

The growth of $V_\mu$ leads to a weighted quadratic-form domain. This specifies the operator realization used for the differentiated equation and the full kernel identity.

\begin{definition}[Closed-form realization]\label{def:linearized-form}
Set
\begin{equation*}
 \cQ_\mu=\left\{\phi\in H^1(\R^N):
       \int_{\R^N}(V_\mu)_+\phi^2<\infty\right\}
\end{equation*}
and
\begin{equation*}
 \mathfrak q_\mu[\phi,\psi]
 =\int_{\R^N}\nabla\phi\cdot\nabla\psi+\int_{\R^N} V_\mu\phi\psi,
 \qquad \phi,\psi\in\cQ_\mu.
\end{equation*}
The operator $L_\mu$ is defined by
\begin{align}
 D(L_\mu)=\{\phi\in\cQ_\mu:\;&\text{there is }f\in L^2
 \text{ such that }\mathfrak q_\mu[\phi,\psi]
   =(f,\psi)_{L^2}\notag\\[-1mm]
 &\text{for every }\psi\in\cQ_\mu\},\qquad L_\mu\phi=f.
 \label{eq:operator-domain}
\end{align}
Its radial restriction is the operator associated with the restriction
of the form to $\cQ_\mu\cap L^2_{\rad}$.
\end{definition}

\begin{proposition}[Self-adjoint realization and characterization of the kernel]\label{prop:form-realization}

\label{prop:self-adjoint-kernel}
If $0<\mu<\mu_{\mathrm{uni}}$, and $\cQ_\mu$, $\mathfrak q_\mu$, and $L_\mu$
be defined as in Definition~\ref{def:linearized-form}. Then the following statements hold.

\begin{enumerate}[\rm{(}i\rm{)}]
    \item The quadratic form $\mathfrak q_\mu$ is densely defined, closed, and bounded
    from below on $L^2(\mathbb{R}^N)$. Consequently, the operator $L_\mu$
    defined by \eqref{eq:operator-domain} is self-adjoint.

    \item For $\phi\in L^2(\mathbb{R}^N)\cap H^1_{\mathrm{loc}}(\mathbb{R}^N)$,
    the following are equivalent:
    \begin{equation}\label{eq:weak-operator-kernel}
    \phi\in\ker L_\mu
    \quad\Longleftrightarrow\quad
    -\Delta\phi=a_\mu\phi
    \quad\text{in }\mathcal D'(\mathbb{R}^N).
    \end{equation}
    Moreover, every $L^2$ distributional solution on the right-hand side
    automatically belongs to $D(L_\mu)$, and hence to $\cQ_\mu$.
\end{enumerate}

\end{proposition}

\begin{proof}
Since $V_\mu$ is bounded from below, choose $C_\mu>0$ such that
\[
    V_\mu+C_\mu\ge1
    \qquad\text{in }\mathbb R^N.
\]
Define
\begin{equation}
    \|\phi\|_{\cQ_\mu}^2
    :=
    \|\nabla\phi\|_2^2
    +
    \int_{\mathbb R^N}(V_\mu+C_\mu)\phi^2\,dx.
    \label{eq:form-norm}
\end{equation}
Since $V_\mu+C_\mu\ge1$, this norm controls the $H^1$-norm.

We first show that $\cQ_\mu$ is complete.
Let $(\phi_n)$ be Cauchy with respect to
\eqref{eq:form-norm}. Then $(\phi_n)$ is Cauchy in
$H^1(\mathbb R^N)$, while
\[
    (V_\mu+C_\mu)^{\frac1 2}\phi_n
\]
is Cauchy in $L^2(\mathbb R^N)$.
Hence, for some $\phi\in H^1(\mathbb R^N)$ and
$g\in L^2(\mathbb R^N)$,
\[
    \phi_n\to\phi
    \quad\text{in }H^1,
\]
and
\[
    (V_\mu+C_\mu)^{\frac1 2}\phi_n\to g
    \quad\text{in }L^2.
\]
Passing to a common subsequence converging almost everywhere
shows that
\[
    g=(V_\mu+C_\mu)^{\frac1 2}\phi
    \quad\text{a.e.}
\]
Thus $\phi\in \cQ_\mu$ and
$\phi_n\to\phi$ in the norm \eqref{eq:form-norm}.
Therefore $\cQ_\mu$ is complete.

Moreover,
\[
\begin{aligned}
    \mathfrak q_\mu[\phi,\phi]
    &=
    \|\nabla\phi\|_2^2
    +
    \int_{\mathbb R^N}(V_\mu)_+\phi^2\,dx
    -
    \int_{\mathbb R^N}(V_\mu)_-\phi^2\,dx \\
    &\ge
    -\|(V_\mu)_-\|_\infty\|\phi\|_2^2,
\end{aligned}
\]
so $\mathfrak q_\mu$ is bounded from below. In addition,
\[
    \mathfrak q_\mu[\phi,\phi]
    +
    C_\mu\|\phi\|_2^2
    =
    \|\phi\|_{\cQ_\mu}^2,
\]
and hence the completeness established above shows that
$\mathfrak q_\mu$ is closed.

We next prove density. Let $\eta_R$ be the cutoff introduced
in Section~2.2. For $\phi\in \cQ_\mu$,
\[
    \eta_R\phi\to\phi
    \qquad\text{in }{\cQ_\mu}
\]
as $R\to\infty$. Indeed, dominated convergence applies to the
weighted term and to $(1-\eta_R)\nabla\phi$, while
\[
    \|\phi\nabla\eta_R\|_2
    \le
    \frac{C}{R}\|\phi\|_2
    \longrightarrow0.
\]
For fixed $R$, standard mollification of $\eta_R\phi$ converges
in $H^1$. Since the mollified functions are supported in a
fixed compact set on which $V_\mu$ is bounded, the convergence
also holds in the weighted part of \eqref{eq:form-norm}.
Therefore
\[
    C_c^\infty(\mathbb R^N)
\]
is dense in $\cQ_\mu$, and hence $\mathfrak q_\mu$ is densely defined on
$L^2(\mathbb R^N)$.

The first representation theorem for closed, lower-bounded
quadratic forms \cite[Chapter~VI]{Kato1995} therefore yields the self-adjoint operator
associated with $\mathfrak q_\mu$, which is precisely the operator
$L_\mu$ defined in \eqref{eq:operator-domain}. This proves (i).

We now prove (ii). Let $\phi\in L^2(\mathbb R^N)$ be a
distributional solution of
\[
    -\Delta\phi=a_\mu\phi.
\]
Since $a_\mu\in L^\infty_{\rm loc}(\mathbb R^N)$,
interior elliptic regularity gives
\[
    \phi\in H^2_{\rm loc}(\mathbb R^N)
    \subset H^1_{\rm loc}(\mathbb R^N).
\]

Since $a_\mu$ is bounded from above, choose $C>0$ such that
\[
    b:=C-a_\mu\ge1.
\]
Then
\[
    -\Delta\phi+b\phi=C\phi
    \qquad\text{in }\mathcal D'(\mathbb R^N).
\]
Testing this equation with $\eta_R^2\phi$, justified by the
standard local $H^1$-approximation argument, and applying
Young's inequality to the cross term gives
\[
    \frac12
    \int_{\mathbb R^N}\eta_R^2|\nabla\phi|^2\,dx
    +
    \int_{\mathbb R^N}b\eta_R^2\phi^2\,dx
    \le
    C'\|\phi\|_2^2,
    \qquad R\ge1,
\]
where $C'$ is independent of $R$.
Letting $R\to\infty$ and using Fatou's lemma, we obtain
\[
    \phi\in H^1(\mathbb R^N),
    \qquad
    b^{\frac1 2}\phi\in L^2(\mathbb R^N).
\]
Since
\[
    V_\mu=b-C,
    \qquad
    (V_\mu)_+\le b,
\]
it follows that
\[
    \int_{\mathbb R^N}(V_\mu)_+\phi^2\,dx<\infty.
\]
Hence $\phi\in \cQ_\mu$.

The distributional identity gives
\[
    \mathfrak q_\mu[\phi,\psi]=0
    \qquad
    \text{for every }
    \psi\in C_c^\infty(\mathbb R^N).
\]
By the density of $C_c^\infty(\mathbb R^N)$ in $\cQ_\mu$,
this identity extends to every $\psi\in \cQ_\mu$.
By the definition of $D(L_\mu)$, we conclude that
\[
    \phi\in D(L_\mu),
    \qquad
    L_\mu\phi=0.
\]
Thus $\phi\in\ker L_\mu$.

Conversely, if $\phi\in\ker L_\mu$, then
\[
    \mathfrak q_\mu[\phi,\psi]=0
    \qquad
    \text{for every }\psi\in \cQ_\mu.
\]
Since $C_c^\infty(\mathbb R^N)\subset \cQ_\mu$, testing with
compactly supported smooth functions gives
\[
    -\Delta\phi=a_\mu\phi
    \qquad\text{in }\mathcal D'(\mathbb R^N).
\]
This proves \eqref{eq:weak-operator-kernel}.
\end{proof}

\begin{lemma}[Translation modes]\label{lem:translation-modes}
For every $0<\mu<\mu_{\mathrm{uni}}$, we have
\begin{equation*}
 \partial_{x_j}u_\mu\in\ker L_\mu\qquad(1\le j\le N),
\end{equation*}
and these $N$ functions are linearly independent.
\end{lemma}

\begin{proof}
Local differentiation of the smooth equation gives
$-\Delta(\partial_{x_j}u_\mu)=a_\mu\partial_{x_j}u_\mu$.
The derivative lies in $L^2$ because $u_\mu\in H^1$, and
Proposition~\ref{prop:form-realization} shows that it belongs to the operator kernel.  Since $u_\mu(x)=U_\mu(r)$ and
$U_\mu'(r)<0$ for $r>0$,
\[
 \partial_{x_j}u_\mu=U_\mu'(r)\frac{x_j}{r}.
\]
A vanishing linear combination would imply
$U_\mu'(r)\frac{c\cdot x}{r}=0$ for every $x\ne0$, which forces $c=0$.
\end{proof}

\begin{proposition}[Absence of additional kernel modes]
\label{prop:no-extra-kernel}
There exists $\mu_{\mathrm{nd}}\in(0,\mu_{\mathrm{uni}}]$ such that
\begin{equation*}
 \ker L_\mu
 =\operatorname{span}\{\partial_{x_1}u_\mu,\ldots,
       \partial_{x_N}u_\mu\}
\end{equation*}
for every $0<\mu<\mu_{\mathrm{nd}}$.
\end{proposition}

\begin{proof}
Fix $N$ and $p$ as in the hypotheses, and let $\mu_1>0$ be as in
Propositions~\ref{prop:uniform-local} and \ref{prop:uniform-tail}.
By Lemma~\ref{lem:translation-modes},
\[
 \operatorname{span}\{\partial_{x_1}u_\mu,\ldots,
                       \partial_{x_N}u_\mu\}\subset\ker L_\mu
 \qquad(0<\mu<\mu_{\mathrm{uni}}).
\]
Suppose that the asserted equality does not hold throughout any
sufficiently small parameter interval. We may then recursively choose
a strictly decreasing sequence $\mu_n\downarrow0$, with
$0<\mu_n<\min\{\mu_{\mathrm{uni}},\mu_1,\frac{1}{n}\}$, such that, on writing
$u_n=u_{\mu_n}$,
\[
 T_n:=\operatorname{span}\{\partial_{x_1}u_n,\ldots,
                          \partial_{x_N}u_n\}
 \subsetneq\ker L_{\mu_n}.
\]
The ground state $u_n$ is well defined by
Proposition~\ref{prop:radial-uniqueness}.
The $N$ translation modes are linearly independent by
Lemma~\ref{lem:translation-modes}, so $T_n$ is a finite-dimensional
closed subspace of $L^2(\R^N)$. Let $P_n$ be the $L^2$-orthogonal
projection onto $T_n$, and choose
$\psi_n\in\ker L_{\mu_n}\setminus T_n$.
Both $\psi_n$ and $P_n\psi_n$ belong to the operator domain and to
its kernel. Since the domain is linear,
\[
 \phi_n:=\frac{\psi_n-P_n\psi_n}{\|\psi_n-P_n\psi_n\|_2}
 \in\ker L_{\mu_n}\cap T_n^\perp,
 \qquad \|\phi_n\|_2=1.
\]
The denominator is positive because $\psi_n\notin T_n$.
In particular,
\[
 \int_{\R^N}\phi_n\partial_{x_j}u_n\,dx=0,
 \qquad j=1,\ldots,N.
\]
By Definition~\ref{def:linearized-form},
$\phi_n\in D(L_{\mu_n})\subset\cQ_{\mu_n}\subset H^1(\R^N)$.
Proposition~\ref{prop:form-realization}{\rm(ii)} gives
\[
 -\Delta\phi_n=a_n\phi_n\quad\text{in }\mathcal D'(\R^N),
 \qquad
 a_n:=\log u_n^2+2+\mu_np\,u_n^{p-1}.
\]
The positivity and smoothness of $u_n$, supplied by
Theorem~\ref{thm:existence}, imply
$a_n\in L^\infty_{\loc}(\R^N)$.

\smallskip

The coefficients above are the diagonal case of
Lemma~\ref{lem:difference-coefficients}, since
\eqref{eq:divided-differences} gives
\[
 a_n=D_{\log}(u_n,u_n)+\mu_nD_p(u_n,u_n).
\]
We record explicitly the bounds needed to apply
Lemma~\ref{lem:no-loss}, including their uniformity in $n$.
Proposition~\ref{prop:uniform-local} gives a constant $M\ge1$, independent
of $n$, such that $0<u_n(x)\le M$ for every $x\in\R^N$.
For a nonempty compact set $K\subset\R^N$, set
$m_K=\min_KU_0>0$. The local convergence in that proposition gives
$u_n\to U_0$ uniformly on $K$, and hence $u_n\ge \frac{m_K}{2}$ on $K$
for all sufficiently large $n$. Since the derivative of
$s\mapsto\log s^2$ is $\frac{2}{s}$, the mean-value theorem yields
\[
 \begin{aligned}
 \|a_n-(\log U_0^2+2)\|_{L^\infty(K)}
 &\le \frac{4}{m_K}\|u_n-U_0\|_{L^\infty(K)}
       +\mu_npM^{p-1}
 \longrightarrow0.
 \end{aligned}
\]
Here $m_K$ is fixed before $n$ tends to infinity.
Moreover, monotonicity of $s\mapsto\log s^2$ and of
$s\mapsto s^{p-1}$ on $(0,\infty)$ gives
\[
 a_n(x)\le\log M^2+2+\mu_1pM^{p-1}=:C_0>0
 \qquad(x\in\R^N, n\ge1).
\]
Since $\log U_0^2=N-|x|^2$, we have proved
\begin{equation}\label{eq:kernel-coefficient-limits}
 a_n\longrightarrow a_0:=N+2-|x|^2
 \quad\text{locally uniformly in }\R^N,
 \qquad a_n\le C_0.
\end{equation}

Fix $A>0$ once and for all. Because $p>1$, we can choose
$\delta\in(0,1)$ so small that
\[
 \log\delta^2+2+\mu_1p\delta^{p-1}\le-A.
\]
Indeed, the logarithmic term tends to $-\infty$ as
$\delta\downarrow0$, while $\delta^{p-1}\to0$.
By Proposition~\ref{prop:uniform-tail}, specifically
\eqref{eq:uniform-algebraic-main}, there is $C_{\mathrm{rad}}>0$,
independent of $n$, such that
\[
 u_n(x)\le C_{\mathrm{rad}}|x|^{-\frac{N-1}{2}}
 \qquad(|x|\ge1).
\]
As $N\ge2$, choose $R\ge1$ with
$C_{\mathrm{rad}}R^{-\frac{N-1}{2}}\le\delta$.
Then $u_n(x)\le\delta$ for $|x|\ge R$, and therefore
\begin{equation}\label{eq:kernel-uniform-negative-tail}
 a_n(x)\le\log\delta^2+2+\mu_np\delta^{p-1}
 \le-A\qquad(|x|\ge R, n\ge1).
\end{equation}
Thus $C_0$, $A$, and $R$ can all be fixed independently of the
sequence index and of the kernel functions $\phi_n$.

\smallskip
We have checked every hypothesis of Lemma~\ref{lem:no-loss}:
$\phi_n\in L^2\cap H^1_{\loc}$, $\|\phi_n\|_2=1$,
$a_n\in L^\infty_{\loc}$, the distributional equation holds, and
\eqref{eq:kernel-coefficient-limits} and
\eqref{eq:kernel-uniform-negative-tail} give the two bounds with
fixed constants. Its cutoff argument establishes
$a_n\phi_n^2\in L^1(\R^N)$ even though $a_n$ may be unbounded below.
The energy identity \eqref{eq:linear-energy-identity} therefore gives
\begin{equation}\label{eq:kernel-uniform-H1}
 0\le\|\nabla\phi_n\|_2^2
 =\int_{\R^N}a_n\phi_n^2\,dx
 \le C_0\|\phi_n\|_2^2=C_0,
 \qquad \|\phi_n\|_{H^1}^2\le1+C_0.
\end{equation}
For completeness, put $b_n=\int_{B_R}\phi_n^2\,dx$.
Splitting this energy identity over $B_R$ and its complement gives
\[
 0\le\|\nabla\phi_n\|_2^2
 \le C_0b_n-A(1-b_n)=(C_0+A)b_n-A.
\]
Here $1-b_n=\int_{\R^N\setminus B_R}\phi_n^2\,dx$ by the normalization.
Rearranging recovers the lower bound
\eqref{eq:local-mass-lower}:
\begin{equation}\label{eq:kernel-fixed-local-mass}
 \int_{B_R}\phi_n^2\,dx\ge\frac{A}{A+C_0}>0
 \qquad(n\ge1).
\end{equation}

The bound \eqref{eq:kernel-uniform-H1} and weak compactness of the
Hilbert space $H^1(\R^N)$ give a subsequence and
$\phi\in H^1(\R^N)$ such that $\phi_n\rightharpoonup\phi$ in $H^1$.
For each positive integer $m$, the restrictions to $B_m$ are bounded
in $H^1(B_m)$, and the Rellich--Kondrachov theorem applies on this
bounded smooth domain to give compactness in $L^2(B_m)$.
A diagonal extraction, with the local limits identified by the
weak $H^1$ limit, therefore gives
\[
 \phi_n\rightharpoonup\phi\quad\text{in }H^1(\R^N),
 \qquad
 \phi_n\to\phi\quad\text{in }L^2_{\loc}(\R^N).
\]
The continuous embedding $H^1(\R^N)\hookrightarrow L^2(\R^N)$
also gives weak convergence in $L^2(\R^N)$ and
$\|\phi\|_2\le1$ by weak lower semicontinuity.
In particular, strong convergence on any $B_m$ with $m>R$ implies
strong convergence on $B_R$, and
\begin{equation}\label{eq:kernel-limit-nonzero}
 \int_{B_R}\phi^2\,dx
 =\lim_{n\to\infty}\int_{B_R}\phi_n^2\,dx
 \ge\frac{A}{A+C_0}>0.
\end{equation}
Thus the global $H^1$ weak limit is nonzero.

\smallskip

Let $\zeta\in C_c^\infty(\R^N)$ and set
$K=\operatorname{supp}\zeta$. The distributional equation gives
\[
 \int_{\R^N}\nabla\phi_n\cdot\nabla\zeta\,dx
 =\int_{\R^N}a_n\phi_n\zeta\,dx.
\]
The left-hand side converges to
$\int_{\R^N}\nabla\phi\cdot\nabla\zeta\,dx$ by weak $H^1$
convergence. For the right-hand side, the decomposition
\[
 a_n\phi_n-a_0\phi
 =(a_n-a_0)\phi_n+a_0(\phi_n-\phi)
\]
and the Cauchy--Schwarz inequality give
\[
 \begin{aligned}
 \left|\int_{\R^N}(a_n\phi_n-a_0\phi)\zeta\,dx\right|
 &\le \|a_n-a_0\|_{L^\infty(K)}
        \|\phi_n\|_{L^2(K)}\|\zeta\|_2\\
 &\quad+\|a_0\|_{L^\infty(K)}
        \|\phi_n-\phi\|_{L^2(K)}\|\zeta\|_2
 \longrightarrow0.
 \end{aligned}
\]
Indeed, the first coefficient tends to zero by
\eqref{eq:kernel-coefficient-limits}, $\|\phi_n\|_{L^2(K)}\le1$,
$a_0$ is bounded on $K$, and the second $L^2(K)$ difference tends
to zero by local strong convergence. Consequently,
\[
 \int_{\R^N}\nabla\phi\cdot\nabla\zeta\,dx
 =\int_{\R^N}(N+2-|x|^2)\phi\zeta\,dx
 \qquad(\zeta\in C_c^\infty(\R^N)),
\]
that is,
\[
 L_0\phi=-\Delta\phi+(|x|^2-N-2)\phi=0
 \qquad\text{in }\mathcal D'(\R^N).
\]
We now use Lemma~\ref{lem:gaussian-kernel}{\rm(ii)}, whose hypotheses
are precisely global $H^1$ membership and this distributional
identity. Its form-domain argument places $\phi$ in the operator
kernel, and the kernel characterization gives
\begin{equation}\label{eq:kernel-limit-span}
 \phi\in\ker L_0
 =\operatorname{span}\{\partial_{x_1}U_0,\ldots,
                       \partial_{x_N}U_0\}.
\end{equation}

\smallskip

Since $\mu_n\downarrow0$ and $u_n\in\cG_{\mu_n}^{\rad}$,
Proposition~\ref{prop:every-selection}, specifically
\eqref{eq:sequence-W-convergence}, gives $u_n\to U_0$ strongly in $W$.
The norm in Definition~\ref{def:energy-space} controls the $H^1$
norm, so, for each $j=1,\ldots,N$,
\[
 \|\partial_{x_j}u_n-\partial_{x_j}U_0\|_2
 \le\|u_n-U_0\|_{H^1}\le\|u_n-U_0\|_W\longrightarrow0.
\]
Using $\|\phi_n\|_2=1$, the Cauchy--Schwarz inequality, and weak
$L^2$ convergence against the fixed function
$\partial_{x_j}U_0\in L^2(\R^N)$, we obtain
\[
 \begin{aligned}
 &\left|\int_{\R^N}\phi_n\partial_{x_j}u_n\,dx
          -\int_{\R^N}\phi\partial_{x_j}U_0\,dx\right|\\
 &\quad\le\|\phi_n\|_2
             \|\partial_{x_j}u_n-\partial_{x_j}U_0\|_2
       +\left|\int_{\R^N}(\phi_n-\phi)\partial_{x_j}U_0\,dx\right|
 \longrightarrow0.
 \end{aligned}
\]
By the preceding argument, the first integral on the left vanishes for every $n$.
Therefore
\[
 \int_{\R^N}\phi\partial_{x_j}U_0\,dx=0,
 \qquad j=1,\ldots,N.
\]
By \eqref{eq:kernel-limit-span}, there are real numbers
$c_1,\ldots,c_N$ with $\phi=\sum_{j=1}^Nc_j\partial_{x_j}U_0$.
Taking the $L^2$ inner product with $\phi$ gives
\[
 \|\phi\|_2^2
 =\sum_{j=1}^Nc_j\int_{\R^N}\phi\partial_{x_j}U_0\,dx=0,
\]
which contradicts \eqref{eq:kernel-limit-nonzero}.
This contradiction excludes failure at arbitrarily small positive
parameters. Hence there exists
$\mu_{\mathrm{nd}}\in(0,\min\{\mu_{\mathrm{uni}},\mu_1\}]$
such that the asserted full kernel identity holds for every
$0<\mu<\mu_{\mathrm{nd}}$.
\end{proof}

\begin{remark}[Spherical-harmonic interpretation]\label{cor:spherical-modes}
The radial potential $V_\mu$ ensures that the closed form decomposes into orthogonal spherical-harmonic sectors. For $0<\mu<\mu_{\mathrm{nd}}$,
Proposition~\ref{prop:no-extra-kernel} identifies the kernel with the
translation derivatives $U_\mu'(r)\frac{x_j}{r}$, whose angular factors are spherical harmonics of degree one. Hence the degree-one kernel consists exactly of these
translations, while the radial sector and every sector of degree
$\ell\ge2$ have zero kernel.
\end{remark}

\subsection{Uniqueness of positive ground states up to translations}
\label{subsec:uniqueness-translations}

\begin{proof}[Proof of Theorem~\ref{thm:nondegeneracy}]
Take $\mu_0=\mu_{\mathrm{nd}}\le\mu_{\mathrm{uni}}$ and fix
$0<\mu<\mu_0$. Proposition~\ref{prop:radial-uniqueness} gives the unique centered radial ground state $u_\mu$.
For $u\in\cG_\mu^+$, Theorem~\ref{thm:existence} gives a center
$x_0$ such that $v(x)=u(x+x_0)$ is positive and radial about the origin.

Translation invariance preserves the space \(W\), the equation, and the action.
Hence \(v\) is again a positive ground state and, by construction, is radial about the origin;
thus \(v\in \cG_\mu^{\mathrm{rad}}\). Proposition~\ref{prop:radial-uniqueness} therefore gives \(v=u_\mu\).
Conversely, every translate of \(u_\mu\) is a positive ground state.
This proves \eqref{eq:ground-state-orbit}.

Proposition~\ref{prop:no-extra-kernel} gives
\[
\ker L_\mu
 =\operatorname{span}\{\partial_{x_1}u_\mu,\ldots,\partial_{x_N}u_\mu\}.
\]
Since
\[
\sum_{j=1}^N a_j\partial_{x_j}u_\mu
 =U_\mu'(r)\frac{a\cdot x}{r},
\]
and \(U_\mu'(r)<0\) for \(r>0\), no nonzero element of this span is radial.
Consequently,
\[
\ker\bigl(L_\mu|_{L^2_{\mathrm{rad}}}\bigr)=\{0\}.
\]
\end{proof}

\section*{Acknowledgments}
This work was partially supported by the National Natural Science Foundation of China (Grant No.~11871389), the Fundamental Research Funds for the Central Universities (Grant No.~xzy012026055), and the Shaanxi Fundamental Science Research Project for Mathematics and Physics (Grant No.~22JSY032).

\section*{Declarations}

\noindent\textbf{Data Availability.}
The manuscript has no associated data.

\medskip

\noindent\textbf{Conflict of interest.}
The authors declare that they have no conflict of interest.

\appendix
\refstepcounter{section}
\section*{Appendix}\label{sec:appendix}
The following estimates give the uniform $L^\infty$ bound and the
H\"older control at infinity used in Section~\ref{sec:limit}.

\begin{lemma}[Uniform Moser iteration]\label{lem:moser-iteration}

Fix \(\bar\mu>0\) and \(M_H>0\). Then there exists
\(C=C(N,p,\bar\mu,M_H)>0\) such that every nonnegative weak solution
\(u\in W\) of \eqref{eq:main} satisfying
\[
0\le\mu\le\bar\mu,
\qquad
\|u\|_{H^1}\le M_H,
\]
also satisfies
\begin{equation}\label{eq:Moser-conclusion}
 \norm{u}_\infty\le C.
\end{equation}

\end{lemma}

\begin{proof}
Suppose first that $N\ge3$.  Choose
\begin{equation}\label{eq:Moser-rho}
 p<\rho<2^*-1.
\end{equation}
For $s\ge0$, the nonlinearity satisfies the upper bound
\begin{equation*}
 s\log s^2+\mu s^p\le C(s+s^\rho),
\end{equation*}
where $C=C(p,\rho,\bar\mu)$.  Let $\beta\ge1$,
$u_T=\min\{u,T\}$,
\[
 w_T=u\,u_T^{\beta-1},\qquad
 \varphi_T=u\,u_T^{2\beta-2}.
\]
For fixed $T>0$, the chain rule gives
$w_T,\varphi_T\in H^1(\R^N)$.  Since
$0\le\varphi_T\le T^{2\beta-2}u$, the monotonicity and
$\Delta_2$ scaling property of $A$ give
$A(\varphi_T)\le C_{T,\beta}A(u)$.  Thus
$\varphi_T\in H^1(\R^N)\cap L^A(\R^N)=W$
and is an admissible test function; all
integrands below are nonnegative or integrable.  Differentiating separately on $\{u<T\}$ and $\{u\ge T\}$ gives
\begin{equation*}
 \int_{\R^N}\nabla u\cdot\nabla\varphi_T
 \ge\frac{2\beta-1}{\beta^2}\int\abs{\nabla w_T}^2.
\end{equation*}
Testing the subsolution inequality and using
$\frac{2\beta-1}{\beta^2}\ge\frac 1 \beta$ yields, after increasing the constant,
\begin{equation}\label{eq:Moser-Sobolev-start}
 \norm{w_T}_{2^*}^2
 \le C\beta^2\left(\norm{w_T}_2^2+
           \int u^{\rho-1}w_T^2\right).
\end{equation}

Set
\begin{equation*}
 s_0=\frac{2^*}{\rho-1},\qquad
 t=\frac{2s_0}{s_0-1},\qquad
 \theta=\frac{(N-2)(\rho-1)}4.
\end{equation*}
Condition \eqref{eq:Moser-rho} gives $0<\theta<1$, and
$2<t<2^*$.  
Since
\[
\frac{\rho-1}{2^*}+\frac{2}{t}=1,
\]
Hölder's inequality and the uniform \(H^1\)-bound give
\[
\int_{\mathbb R^N}u^{\rho-1}w_T^2
\le C\|w_T\|_t^2.
\]
Moreover,
\[
\frac1t=\frac{1-\theta}{2}+\frac{\theta}{2^*},
\]
so interpolation yields
\[
\|w_T\|_t^2
\le
\|w_T\|_2^{2(1-\theta)}
\|w_T\|_{2^*}^{2\theta}.
\]

Young's inequality absorbs one half of the last
$L^{2^*}$ term in \eqref{eq:Moser-Sobolev-start} and gives
\begin{equation*}
 \norm{w_T}_{2^*}^2
 \le C_0\beta^\kappa\norm{w_T}_2^2,
 \qquad \kappa=\frac{2}{1-\theta}.
\end{equation*}
Monotone convergence as $T\to\infty$ yields
\begin{equation*}
 \norm{u}_{2^*\beta}
 \le(C_0\beta^\kappa)^{\frac 1{2\beta}}
       \norm{u}_{2\beta}.
\end{equation*}
 Starting with $\beta_0=1$ and
 $\beta_{k+1}=\frac {2^*}{2}\beta_k$, iteration gives
\[
 \norm{u}_{2\beta_{k+1}}
 \le \norm{u}_2\prod_{j=0}^k
 (C_0\beta_j^\kappa)^{\frac 1{2\beta_j}}.
 \]
The logarithm of the product is
bounded by
\[\frac12\sum_j\beta_j^{-1}(\log C_0+\kappa\log\beta_j)<\infty.\]
Letting $k\to\infty$ proves \eqref{eq:Moser-conclusion}.

For $N=2$, choose a fixed $Q>2$ and then $\rho>p$.  The Sobolev
inequality
\[
 \norm{z}_Q^2\le C_Q\bigl(\norm{\nabla z}_2^2+\norm z_2^2\bigr)
\]
replaces the Sobolev inequality with exponent $2^*$ in
\eqref{eq:Moser-Sobolev-start}.  Choose
$ s_0>\max\{\frac{Q}{Q-2},\,\frac{2}{\rho-1}\}$ and put
$t=\frac{2s_0}{s_0-1}\in(2,Q)$. Then $(\rho-1)s_0>2$.  The coefficient
$u^{\rho-1}$ is bounded in $L^{s_0}$ because
$H^1(\R^2)$ embeds into $L^m$ for every $2\le m<\infty$.  Interpolation between
$L^2$ and $L^Q$, followed by Young's inequality, gives
\[
 \norm{u}_{Q\beta}\le(C_0\beta^\kappa)^{\frac 1{2\beta}}
        \norm{u}_{2\beta}.
\]
Taking $\beta_{k+1}=\frac{Q}{2}\beta_k$ gives the same convergent product.
Fixing the auxiliary exponents \(Q,\rho,s_0\) once and for all in terms of
\(N\) and \(p\), all constants depend only on
\(N,p,\bar\mu,M_H\). This proves Lemma~\ref{lem:moser-iteration}.
\end{proof}

The next lemma combines the modulus of continuity of $s\log s^2$ at
zero with interpolation to obtain the H\"older estimate used in
Proposition~\ref{prop:uniform-tail}.

\begin{lemma}[H\"older estimates for the logarithmic nonlinearity near zero]
\label{lem:holder-tail}

Define
\[
h(s)=
\begin{cases}
s\log s^2, & s>0,\\
0, & s=0.
\end{cases}
\]
Then the following statements hold.

\begin{enumerate}[\rm{(}i\rm{)}]
    \item For every $M>0$ and every $\sigma\in(0,1)$, there exists
    $C=C(M,\sigma)>0$ such that
  \begin{equation}\label{eq:h-holder}
 \abs{h(t)-h(s)}\le C(M,\sigma)\abs{t-s}^\sigma
 \qquad(0\le s,t\le M).
\end{equation}

    \item Let $0<\alpha<\sigma<1$, and let
    $f\in C^{0,\sigma}(B)$, where $B\subset\mathbb{R}^N$
    has bounded diameter. Then
   \begin{equation}\label{eq:holder-interpolation}
 [f]_{C^{0,\alpha}(\overline B)}
 \le C_{\alpha,\sigma,B}
 \norm f_\infty^{1-\frac{\alpha}{\sigma}}
 [f]_{C^{0,\sigma}(\overline B)}^{\frac{\alpha}{\sigma}}.
\end{equation}

    \item Fix $R>0$, $\bar\mu>0$, and $0<\alpha<1$.
    Let $y_j\in\mathbb{R}^N$, $0\le\mu_j\le\bar\mu$, and
    \[
    u_j\ge0
    \qquad\text{on }B_R(y_j).
    \]
    Assume that $\{u_j\}$ is uniformly bounded and uniformly Lipschitz on
    $B_R(y_j)$, and that
    \[
    \|u_j\|_{L^\infty(B_R(y_j))}\longrightarrow0.
    \]
    Then
    \begin{equation}\label{eq:RHS-holder-tail}
 \norm{u_j\log u_j^2+\mu_j u_j^p}_{C^{0,\alpha}(B_R(y_j))}
 \xrightarrow[j\to\infty]{}0
\end{equation}
\end{enumerate}

\end{lemma}

\begin{proof}
By symmetry, it is enough to consider \(t>s\). Set
\[
d=t-s>0.
\] 

If \(s\le d\), then \(s,t\le2d\). For \(0<d\le1\), the estimate
\[
|r\log r^2|\le C_\sigma r^\sigma
\qquad (0\le r\le 2)
\]
gives
\[
|h(t)-h(s)|
\le |h(t)|+|h(s)|
\le C_\sigma d^\sigma.
\]
If \(d\ge1\), the same conclusion follows from the boundedness of \(h\) on
\([0,M]\), after enlarging the constant \(C(M,\sigma)\).

If \(s>d\), then by the mean-value theorem,
\[
|h(t)-h(s)|
\le
d\sup_{r\in[s,t]}|\log r^2+2|.
\]
When \(s\le1\), every \(r\in[s,t]\) satisfies \(r\ge s>d\), and hence
\[
\sup_{r\in[s,t]}|\log r^2+2|
\le C(1+|\log d|).
\]
Therefore
\[
|h(t)-h(s)|
\le C d(1+|\log d|)
\le C_\sigma d^\sigma.
\]
For \(s\ge1\), \(h'\) is bounded on \([1,M]\), and the result follows
immediately. This proves \eqref{eq:h-holder}.

For \eqref{eq:holder-interpolation}, we use
\[
|f(x)-f(y)|
\le
\min\left\{
2\|f\|_\infty,\,
[f]_{C^{0,\sigma}(B)}|x-y|^\sigma
\right\},
\]
and distinguish the two regimes according to whether
\[
|x-y|
\lessgtr
\left(
\frac{2\|f\|_\infty}
{[f]_{C^{0,\sigma}(B)}}
\right)^{\frac{1}{\sigma}}.
\]
This yields the stated exponents; the case of a vanishing seminorm is immediate.

For (iii), fix $\sigma\in(\alpha,1)$.
By \eqref{eq:h-holder} and the uniform Lipschitz bound on \(u_j\),
the functions \(h(u_j)\) are uniformly bounded in \(C^{0,\sigma}\).
Since \(p>1\) and the \(u_j\) are uniformly bounded, the functions
\(u_j^p\) are uniformly Lipschitz as well. Hence
\[
f_j:=h(u_j)+\mu_j u_j^p
\]
is uniformly bounded in \(C^{0,\sigma}\), whereas the uniform convergence
\(u_j\to0\) implies
\[
\|f_j\|_\infty\to0.
\]
Applying \eqref{eq:holder-interpolation} therefore gives
\[
\|f_j\|_{C^{0,\alpha}}\to0,
\]
which is \eqref{eq:RHS-holder-tail}.
\end{proof}
\bibliographystyle{plain}
\bibliography{references}
\end{document}